\documentclass[leqn,12pt]{article}
\usepackage{graphicx}
\usepackage{epsfig}
\usepackage[utf8]{inputenc}
\usepackage[T1]{fontenc}
\usepackage{amsfonts,amsmath}
\usepackage{latexsym}
\usepackage{float} 
\usepackage{multicol}
\usepackage{amssymb}
	\usepackage{theorem}
	\DeclareMathOperator*{\argmax}{arg\,max} 
\usepackage{ulem}
\usepackage[inoutnumbered,ruled]{algorithm2e}             
\usepackage{epstopdf}

\newsavebox{\fmbox}

\def \Sup{\displaystyle\sup}
\def \Max{\displaystyle\max}
\def \Lim{\displaystyle\lim}

\def \Sup{\displaystyle\sup}

\def \Max{\displaystyle\max}
\usepackage{color}
\def\lbr{[\![}

\def \Lim{\displaystyle\lim}

\def \Sup{\displaystyle\sup}

\def \Max{\displaystyle\max}
\newlength{\oldparindent}
\newcommand{\pf}{\hfill$\bullet$}
\usepackage{hyperref}
\usepackage{url}
\usepackage{cleveref}
\title{ The value of the information in the Moral Hazard setting}
\author{Ishak Hajjej\thanks{\small  CREST $\&$  IUT,  LMBA, France, email \href{mailto: ishak.hajjej@ensae.fr}{ ishak.hajjej@ensae.fr}  }, Caroline Hillairet\thanks{\small CREST, ENSAE Paris, France, email \href{mailto: Caroline.Hillairet@ensae.fr}{ Caroline.Hillairet@ensae.fr}},
Mohamed Mnif\thanks{\small  ENIT, LAMSIN, University of Tunis El Manar, Tunis, Tunisia, email \href{mailto:mohamed.mnif@enit.utm.tn}{mohamed.mnif@enit.utm.tn}   } }
\usepackage{subcaption}

\newtheorem{definition}{Definition }[section]
\newtheorem{proposition}[definition]
{Proposition }
{DÃ©finition}
\newtheorem{lemme}[definition]%
{Lemma }
{Lemme }
\newtheorem{theoreme}[definition]%
{Theorem }
{ThÃ©orÃšme }
{Corollary }
\newtheorem{remarque}[definition]%
{Remark }
{Remarque }
{Corollaire}
\newtheorem{hypothese}[definition]%
{Assumption}
{Condition}
\newtheorem{ex}[definition]%
{Example}
\newcommand{\esp}{\mathbb{E}} 

\newcommand{\un}{{\mathbf{1}}} 

\newcommand{\pr}{\mathbb{P}}

\newcommand{\ff}{\mathbb{F}}

\newcommand{\R}{\mathbb{R}}
\newcommand{\N}{\mathbb{N}}

\newcommand{\F}{\mathcal{F}}

\numberwithin{equation}{section}
\usepackage{subcaption}

\begin{document}

\maketitle

\vspace*{-1cm}
\section*{Abstract}
This article studies the problem of evaluating the information that a Principal lacks when establishing an
incentive contract with an Agent whose effort is not observable.
The Principal ("she") pays a continuous rent to the Agent ("he"), while the latter gives a best response characterized by his effort, until a terminal date decided by the Principal when she stops the contract and gives compensation to the Agent. 
The output process of the project  is a diffusion process driven by a Brownian motion whose drift is impacted by the Agent's effort.
The first part of the paper investigates  the  optimal stochastic control problem when the Principal and the Agent share the same information. This situation,  known as the first-best case,
 is solved by tackling the Lagrangian problem.
In the second part, the Principal observes the output process but she may not observe the drift and the Brownian motion separately. This situation is known as the second-best case.
We derive the best response of the Agent, then
we solve the mixed optimal stopping/stochastic control problem of the Principal under a fixed probability and on the  filtration  generated by the Brownian motion, which is  larger than the one generated by the output process (that corresponds to the   information available for the Principal).
Under some regularity conditions, the Principal value function is characterized by solving the associated Hamilton Jacobi Bellman Variational Inequality. At the optimum, we prove that the two filtrations coincide.
Finally, we compute the value of the information for the Principal provided by the observation of the Agent's effort. It is defined as the difference between
the principal value function in the first-best and second-best cases.

{\it Keywords}: Moral Hazard, Principal-Agent, stochastic control, optimal stopping, Hamilton Jacobi Bellman Variational Inequality, first-best case, second-best case .

{\it MSC Classification} : 60G40, 91B40, 91B70, 93E40.

{\it Funding}: This research is supported by a grant of the French National Research Agency (ANR), ''Investissements d'Avenir'' (LabEx Ecodec/ANR-11-LABX-0047).

\section{Introduction}

A Principal-Agent problem concerns two individuals: a Principal (she) and an Agent (he). The Principal proposes a contract, which the Agent is free to accept or refuse,  but neither part can change their mind. The Principal seeks to achieve two objectives. First, she wants the Agent to accept the contract. Second, the Principal wishes to extract the maximum profit from the contract. In game theory, such a situation forms what is called a Stackelberg game where the Principal is the leader  and the Agent is the follower. There are many applications of the Principal-Agent problem, notably in finance (such as El Euch et al. \cite{euch2021optimal} for market regulation), portfolio management (see Cvitanic et al. \cite{cvitanic2017moral}) and in the electricity markets (cf. Alasseur et al. \cite{alasseur2017adverse} and more recently Aïd et al. \cite{aid2022optimal}).  Numerous situations in the economic literature lead to Principal Agent's formulation, such as  Public-Private Partnership (PPP) contracts.  Auriol and Picard \cite{auriol2013theory} discussed the appropriateness of PPP contracts when the public entity and the private operator do not share the same information about the cost parameter during the project life.  Hajjej et al. \cite{hajjej2017optimal} derived the optimal perpetual contract  using techniques of stochastic control under partial information. Hajjej et al.  \cite{hajjej2019optimal} proposed a similar modeling with a continuous payment in random horizon, by adding the possibility of stopping the contract at a random time, decided by the public.\\
In the literature, mainly  three types of contracts are considered,  according to the level of information.
\begin{itemize}
\item The first type is commonly called the Risk-Sharing case or first-best case, in which the Principal and the Agent have access to exactly the same information.
This case  was studied, among others,  by  Borch \cite{borch1992equilibrium},  Arrow \cite{arrow1964role} and  Wilson \cite{wilson1968theory}.  
In the exponential utility for both the Principal and the Agent, 
Muller \cite{muller1997first}  showed that the optimal contract is a linear function of the output process terminal value. Cadenillas et al.  \cite{cadenillas2007optimal} considered the case where the agent can control both the drift and the volatility of the output process.  This work was followed by that of Cvitanic et al.  \cite{cvitanic2004first} by using the maximum principle.
\item The second type is commonly called Moral Hazard or second-best case.  This situation corresponds  to the case where the Principal is unable to observe the action of the agent and she only observes the output process.  The first paper on Principal Agent problems in continuous-time is the one of Holmstrom
and Milgrom \cite{holmstrom1987aggregation}.  They considered a Brownian setting in which the agent controls the drift of the output process, and receives a lumpsum payment at the end of the contract, that is a finite time horizon.   Schattler and Sung \cite{schattler1993first} generalized these results using dynamic programming and martingale methods. Cvitanic et al. \cite{auriol2013theory},  \cite{bellman1957dynamic}   considered a general formulation in which the agent's efforts impact both the drift and the volatility of the output process, using second-order BSDE in a non Markovian stochastic control setting.
\item The third type is commonly called Adverse Selection. This situation corresponds to the case where the principal does not fully know the characteristics of the agent.  This type of problem has been studied by Sung \cite{sung2005optimal} then by Cvitanic and Zhang \cite{cvitanic2012contract}  and Carlier et al. \cite{carlier2007optimal}.
\end{itemize}
Compared to the first-best case,  the second-best case corresponds to a Stackelberg  game between the Principal and the Agent: the Principal proposes a contract,  the Agent gives the best response characterized by his effort.  Then, taking into account the optimal efforts, the Principal calculates the optimal contract that maximizes her utility. Some papers  derived optimal Principal-Agent contracts in both situations of first-best and second-best,   in the case where the Agent receives  a terminal lumpsum payment and not a continuous rent.
Mastrolia and Possamai \cite{mastrolia2018moral} treated a Principal-Agent problem in the case where the Agent controls the drift of the output and with uncertainty on the volatility.  In the first-best case, they showed that the optimal contracts are in a class of contracts linear with respect to the output and its quadratic variation. For the second-best case, they used Backward Stochastic differential equations (BSDE) to obtain a probabilistic representation of the agent value function.  This representation characterizes the optimal action chosen by the agent.  Aïd et al. \cite{aid2022optimal}  applied the Principal Agent problem in electricity markets. They showed that in the first-best case,  the price of energy is a convex combination of the marginal value of energy for the consumer and the marginal cost for the producer.  In the second-best case, the price of energy is non-constant and non-linear and is a decreasing function of time which induces more effort at the beginning of the period than at the end.  In the setting of  exponential utility for  the Principal and the Agent,  Muller \cite{muller1997first} showed that the optimal contract in the first-best  case is linear in the final value of the output as in the  second-best case. \\

The Principal faces a lack of information in the second-best case since she can not observe the effort of the Agent.
Quantifying the value of an additional information has been investigated by  some authors.   Amendinger et al. \cite{amendinger2003monetary} considered a risk averse investor who maximizes his expected utility from terminal wealth with  decisions based on the available information flow. This investor faces the opportunity to acquire some additional initial information. The value of this information is defined as the amount of money that he can pay for the additional information such that the investor is indifferent in the sense that his expected utility under optimal control is unchanged between paying nothing and not having the additional information and paying something and having the additional information. In other words, the value of the information  is balanced out by the informational advantage in terms of maximal expected utility. In our case, the information is related to the knowledge of the effort of the Agent.
The value of the information could be defined by the monetary amount that the Principal accepts to receive to compensate the lack of information. As the Principal is risk neutral, i.e. her utility function is  linear, the monetary value of the information is then defined as the difference between the value function of the Principal in the first-best case and in the second-best case.\\

This paper considers a contract in which a risk neutral Principal proposes a continuous payment to a risk averse Agent until a random horizon decided by her to stop the contract. The Agent has the choice between accepting the contract and making an effort or refusing the contract if its reservation constraint  is not satisfied. This work is related to the literature on dynamic contracting problems in continuous time. 
The first part of the paper is dedicated to the first-best case in which  the Principal and the Agent have the same information: the two parts share the risk between themselves.
The situation leads to a stochastic control problem with optimal stopping for a single individual "the Principal", who chooses both the rent and the efforts. The second part of the paper is dedicated to  the second-best case.
In particular, the Principal may not observe the effort of the Agent but only its impact on output process. In the literature, this problem is usually tackled using a weak approach :  the output process follows a Brownian motion under a reference probability measure and the Agent changes the distribution of the output process by changing its drift; this induces a new probability measure,  depending  on the effort of the Agent,  under which  the problem is studied. For a given contract,
the Agent optimizes his criterion on this  probability  measure, leading to  the best response (effort) of the Agent.
Then the Principal solves her problem under the law induced by the best response of the Agent, this is a standard mixed optimal stopping/stochastic control problem.
Contrary to this  weak formulation, we adopt  in this paper a strong formulation.  We fix a probability space, the Principal observes the output process, but she may not observe the drift nor the Brownian motion. The Principal and the Agent optimize their criteria under different  information flow, the Agent's filtration being larger than the  Principal's filtration.  We are in the context of stochastic control under partial observation (see Bensoussan \cite{bensoussan2004stochastic}). 
We solve the  mixed optimal stopping/stochastic control problem of the Principal on a larger set of controls that correspond to the adapted controls with respect to the total filtration which contains all information of the Principal and the Agent. At the optimum, we prove that  the total filtration and the filtration generated by the output process coincide. \\

There are two contributions in this paper. First, we  solve the Principal Agent Problem by using the strong approach in the non degenerate case i.e. the volatility of the output process is positive. Second,  we  compute the value of the information for the Principal provided by the observation of the Agent's effort.
The outline of the paper is as follows. In Section 2, we formulate the problem, using the strong 
approach and we define the Principal and the Agent problems, both in the first-best and second-best case.
Section 3 is dedicated to  the Principal-Agent problem in the first-best framework.
In Section 4, we solve the Principal-Agent problem in the second-best case by determining
the incentive compatible contract for the Agent and deriving the associated Hamilton Jacobi Bellman
Variational Inequality associated to the Principal value function.
Section 5 is dedicated to the numerical study.

 \section{Formulation of the Principal and Agent problem}
 Throughout the paper, let $\left(\Omega, \mathbb{F}=(\mathcal{F}_t)_{t\geq0}, \pr \right)$ be a filtered probability space where  $\mathbb{F}=(\mathcal{F}_t)_{t\geq0}$  is  the  filtration (satisfying the usual conditions of right-continuity and completeness) generated by $W$ an one dimensional Brownian motion.  We consider a time-continuous Principal-Agent problem where a risk-neutral Principal proposes a continuous
payment to a risk-averse Agent until a  random  terminal horizon  decided by the Principal.
In return the Agent, if accepting the contract,  makes an effort impacting the drift of the output  process of the project. 
 Throughout the paper, we adopt  the strong formulation that is the Principal-Agent problem is solved under the reference probability measure $\pr$.
 A contract is a triplet $\Gamma=((R_t)_t,\tau,\xi)$ where $R$ is a non-negative rent,  $\tau$ is the random terminal date of the contract (decided by the Principal) and 
$\xi$ is the cost of stopping the contract. The measurability of the triplet will be precised hereafter, depending on the situation considered (first-best or second-best case).
Given a contract $\Gamma=((R_s)_{s},\tau,\xi)$ offered by the Principal, the Agent gives a best response in terms of an effort  $(A_s)_{s\geq 0}$  which is an non-negative $\mathbb{F}$-progressively measurable process. The Agent's effort impacts 
the output process  of the project, which is an $\ff$-progressively measurable process $(X_s)_{s\geq 0}$ given by
\begin{equation}\label{Sv}
X_t := X_0+\int_0^t \varphi(A_s)ds+\sigma W_t,
\end{equation}
where
$X_0>0$ is the initial value of  the output process.
and  $\sigma>0$ is the volatility (or diffusion coefficient), that is assumed to be constant.
\noindent We introduce the following notations:
\begin{itemize}
\item $\ff^X:=(\mathcal{F}_t^X)_{t \geq 0}$ is the filtration generated by the output process $X$.
\item ${\cal{T}}$ is the set of all  {$\ff$}-stopping times.
\item ${\cal{T}}^X$ is the set of all  {$\ff^X$}-stopping times.
\end{itemize}
For the first-best case, the Principal and the Agent share the same information. A contract can be gathered into a quadruplet $((R_t)_t,\tau,\xi, A)$,
 where the rent $R$ is is a non-negative ${\ff}${-progressively measurable process}, $\tau\in{\cal{T}}$, and  
$\xi$ is  an non-negative ${\F}_\tau$-measurable {random variable}.
For the second-best case, the Principal observes the output process $X$, but she does not observe directly the Agent's effort.
We are in the context of asymmetric information. Thus the information of the Principal is conveyed by the filtration $\ff^X$. A contract is then a triplet $\Gamma=((R_t)_t,\tau,\xi)$ where $R$ is a non-negative ${\ff}^X${-progressively measurable process}, $\tau\in{{\cal T}^X}$, and
$\xi$ is  a non negative ${\F}^X_\tau$-measurable {random variable}\footnote{This is in contrast with the weak formulation, in which  the  rent $R$ is an $\ff$-progressively measurable process, $\tau\in{\cal{T}}$, and  
$\xi$ is   ${\F}_\tau$-measurable.}
Given a contract $\Gamma=((R_s)_{s},\tau,\xi)$ offered by the Principal, the Agent gives a best response in terms of an effort
process $A$: this is a Stackelberg leadership model.
 The Agent accepts the contract only if his expected payoff  at time $0$ exceeds his reservation value $\underline{x}>0.$\\
Let us now define the functions involved in the formulation of the optimization problems:

\begin{hypothese}\label{hyp}
\begin{itemize}
\item[\pf]$\varphi$ is the function that models the marginal impact of the Agent's efforts on the output process $\varphi:[0,\infty) \rightarrow[0,\infty)$ is {$C^2$} { strictly} concave, bounded, increasing, {$\varphi(0)=0$} and $\varphi'(0)>0$. We denote by $\|\varphi \|_{\infty}:=\sup_{a \geq 0 }\varphi(a)$.  
\item [\pf] The utility function of the Agent {$U :[0,\infty) \rightarrow[0,\infty)$} is $C^2$ strictly concave increasing and satisfying $U(0)=0$ and  Inada's conditions  $U'(\infty):=\Lim_{x\longrightarrow \infty} U'(x)= 0,~U'(0):=\Lim_{x\longrightarrow 0} U'(x)=\infty$.
\item [\pf] $h$ is the cost of the effort for the Agent; $h:[0,\infty) \rightarrow[0,\infty)$ is {$C^2$}, {strictly} convex {increasing}, $h(0)=0$ and $h^{'}(0)>0$.
\item[\pf] The  time preference parameter of the Agent (resp. of the Principal) is a positive constant denoted $\lambda$ (resp.  $\delta$). As the Principal is usually less  impatient than the Agent, we assume that $\lambda \geq \delta$.
  \end{itemize}
\end{hypothese}
\noindent {\underline {\it{Admissible Contracts:}}}\\
We define the following sets of admissible contracts, depending on the information flow which will be used to solve the optimization problems of the Principal and the Agent in the first-best and second-best cases. \\
\noindent For the first-best case, and as the Principal and the Agent share the same information, the set of admissible strategies is defined as follows:
\begin{eqnarray}\label{ensfb}
  \mathcal{E}&=&\{((R_s)_{s\geq 0},\tau,\xi, (A_s)_{s\geq 0})\mbox{ such that}~R_s\geq0\mbox{ and }A_s\geq0~ds\otimes d\pr\mbox{ are }\\
   & &\mathbb{F}\mbox{-progressively measurable},~\tau\in \mathcal{T}, \xi\geq 0~\mbox{is}~\mathcal{F}_\tau~\mbox{measurable}, \nonumber\\
& & \esp \left[\int_0^\infty e^{-\delta s}|\varphi(A_s)|\vee e^{-\lambda s}|h(A_s)| ds\right]<\infty,
\esp \left[\int_0^\infty e^{-\lambda s} |U(\R_s)| \vee e^{-\delta s} |R_s| ds\right]<\infty,\nonumber\\
  &&\mbox{ and }\esp \left[e^{-\lambda \tau} U(\xi)\vee e^{-\delta \tau} |\xi|{\un_{\{\tau<\infty\}}}\right]<\infty\}.\nonumber
\end{eqnarray}
\noindent For the second-best case, we fix $\rho>0$, and we introduce the following sets of admissible strategies for the Agent and for the Principal:
\begin{eqnarray}\nonumber
\mathcal{D}^{Ag}_\rho&:= &\{ (A_s)_{s\geq 0}~\ff~\mbox{-progressively measurable},~A_s\geq0~ds\otimes d\pr ~a.e~\mbox{such that}~\nonumber
\\
& &\esp \left[\int_0^\infty e^{(\rho-2\lambda) s}|h(A_s)|^2ds\right]<\infty ~\mbox{and}~\esp \left[\int_0^\infty e^{(\rho-2\delta)s}|\varphi(A_s)|^2\right]<\infty\}\label{calA-4}
\\\nonumber
\mathcal{D}^P_\rho&:=&\{ \big((R_s)_{s\geq 0},\tau,\xi\big), \mbox{such that}~ R\geq0~\mbox{ is } \mathbb{F}^X\mbox{-progressively measurable},~\tau\in\mathcal{T}^X\\
& &\xi\geq 0~\mbox{is}~\mathcal{F}_\tau^X~\mbox{measurable}~\mbox{such that }\esp \left[\int_0^\infty e^{(\rho-2\lambda )s } |U(R_s)|^2\vee e^{(\rho-2\delta) s} |R_s|^2 ds\right]<\infty\nonumber \\
& &\mbox{and }
\esp \left[(e^{(\rho-2\lambda) \tau} U^2(\xi)\vee e^{(\rho-2\delta) \tau} |\xi|^2){\un_{\{\tau<\infty\}}}\right]<\infty\}.\label{calA-44}
\end{eqnarray}
\noindent {\underline {\it{Objective function for the Agent and the Principal:}}}\\
Given an admissible strategy $(\Gamma, (A_s)_{s\geq 0})$, where $\Gamma=((R_s)_{s\geq 0},\tau,\xi)$, the objective function of the risk-averse Agent is defined by the expectation {under the probability $\pr$ of his aggregate utility of the rent minus the cost of his effort plus the utility of the penalty when  the  contract is stopped, all these quantities being discounting using the Agent's time preference parameter $\lambda$:
\begin{equation}\label{Objconsstatic}
J_0^{Ag}(\Gamma,A) :=\esp \left[\int_0^\tau e^{-{{{\lambda}}s}}(U(R_s)-h(A_s))ds+e^{-{{\lambda}}\tau}U(\xi){\un_{\{\tau<\infty\}}}\right].
\end{equation} 
Using the Agent's information flow $\ff$, his objective function starting from time $t$ is given by:
\begin{equation}\label{Objcons}
J_t^{Ag}(\Gamma,A) :=\esp \left[\int_t^\tau e^{-{{{\lambda}}(s-t)}}(U(R_s)-h(A_s))ds+e^{-{{\lambda}}(\tau-t)}U(\xi){\un_{\{\tau<\infty\}}}\big| \F_t\right].
\end{equation} 
The objective function of the risk-neutral Principal is defined by the expectation under the probability $\pr$ of the output process minus the rent paid to Agent minus the cost of stopping the {contract},  all these quantities being discounting using the Agent's time preference parameter $\delta$: 
\begin{eqnarray*}\label{inipublic}
J_0^{P}(\Gamma,A)	&:=&\esp\left[\int_0^\tau e^{-\delta s}(dX_s-R_s)ds-e^{-\delta \tau} \xi{\un_{\{\tau<\infty\}}}\right]\\
&=&\esp\left[\int_0^\tau e^{-\delta s}(\varphi(A_ s)-R_s)ds-e^{-\delta \tau} \xi {\un_{\{\tau<\infty\}}}\right].
\end{eqnarray*}
Using the Principal's information flow $\ff$ in the first-best case (resp.  $\ff^X$ in the second -best case), her  objective function starting from time $t$ is given by:
\begin{eqnarray*}
J_t^{P,FB}(\Gamma,A)	
&=&\esp\left[\int_t^\tau e^{-\delta (s-t)}(\varphi(A_ s)-R_s)ds-e^{-\delta (\tau-t)} \xi {\un_{\{\tau<\infty\}}} \big| \F_t \right].\\
J_t^{P,SB}(\Gamma,A)	
&=&\esp\left[\int_t^\tau e^{-\delta (s-t)}(\varphi(A_ s)-R_s)ds-e^{-\delta (\tau-t)} \xi {\un_{\{\tau<\infty\}}} \big| \F^X_t \right].\\
\end{eqnarray*}

\noindent{\underline {\it{Value function for the Principal in the first-best case:}}\\
Since the  two parts share the same information, the Principal solves:
\begin{equation*}
V_0^{FB}:= \Sup_{(\Gamma,A)\in\mathcal{E}}J_0^{P}(\Gamma,A)
\end{equation*}
subject to the reservation constraint  $J_0^{Ag}(\Gamma,A)\geq \underline{x}$, where $\mathcal{E}$ is given by (\ref{ensfb}).\\

\noindent{\underline {\it{Value function for  the Principal in the second-best case:}}\\
Given $\Gamma \in \mathcal{D}^{P}$, the Agent solves
\begin{equation*}
\sup_{A\in \mathcal{D}^{Ag}_\rho}J_0^{Ag}(\Gamma,A),
\end{equation*}
where $\mathcal{D}^{Ag}_\rho$ is given by (\ref{calA-4}). 
 The public anticipates the Agent's best response $A^*$ to propose the optimal contract and aims to solve:
\begin{equation*}
V_0^{SB} := \sup_{\Gamma\in \mathcal{D}^{P}}J_0^{P}(\Gamma,A^*),
\end{equation*}
subject to the reservation constraint $ J_0^{Ag}(\Gamma,A^*)\geq \underline{x}$, where $\mathcal{D}^{P}$ is defined by (\ref{calA-44}).

\begin{remarque}
  The difference between the strong formulation and the weak formulation lies in the conditions of measurability.
  In the weak formulation, the objective function for the Agent and the Principal are computed under the probability measure induced
  by the effort of the agent (usually denoted by $\pr^A$) {and} are both conditionally on $\F_t$, while in the strong formulation, the objective functions are computed under $\pr$  the objective function of Agent is, conditionally on $\F_t$ and the objective function of the Principal is, conditionally on $\F_t^X$.
\end{remarque}

\section{The first-best case}
In this section, the Principal and the Agent share the same information. They observe both $X$ and the effort $A$. We are in the context of risk sharing.
They have to agree how to share the risk between themselves. We first rewrite the problem of the Principal in a more tractable stochastic control form. Then,
we provide some properties related to the optimal contract.
We distinguish two cases: either the Principal does not stop the contract i.e the infinite horizon case, or she could stop the contract.
\subsection {The infinite horizon contract}
The problem becomes a stochastic control problem for a single individual "the Principal", who chooses both the rent and the effort, under the reservation constraint.
In this case, the Principal's optimization problem is expressed as follows:
\begin{equation}\label{horizoninfini}
v(x)=\Sup_{(R,\infty,0,A)\in\mathcal{E}}\esp\left[\int_0^\infty e^{-\delta s}(\varphi(A_s)-R_s)ds\right]
\end{equation}
subject to the reservation constraint
\begin{equation}\label{Rc}
\esp\left[\int_0^\infty e^{-{{{\lambda}} s}}(U(R_s)-h(A_s))ds\right]\geq \underline{x}
\end{equation}
We introduce a Lagrange multiplier $\lambda_{Lag}$ in order to solve this problem, and we consider the unconstrained problem: 
\begin{eqnarray*}
\Sup_{(R,\infty,0,A)\in\mathcal{E}}\left\lbrace \esp\left[\int_0^\infty e^{-\delta s}(\varphi(A_s)-R_s)ds\right]+\lambda_{Lag}(\esp\left[\int_0^\infty e^{-\lambda s}(U(R_s)-h(A_s))ds\right]-\underline{x})\right\rbrace \\
=\Sup_{(R,\infty,0,A)\in\mathcal{E}}\esp\left[\int_0^\infty e^{-\delta s}(\varphi(A_s)-R_s)ds+\lambda_{Lag}[e^{-\lambda s}(U(R_s)-h(A_s))ds]\right]-\lambda_{Lag}\underline{x}.
\end{eqnarray*}
\begin{proposition}\label{horinfini}
 Under Assumption \ref{hyp}, the solution of (\ref{horizoninfini}) under the reservation constraint (\ref{Rc}) is given by\footnote{using the notation $x\vee y:= \max(x,y).$}: 
 \begin{equation}\label{effortrent}
R_s^*=(U')^{-1}(\frac{1}{\lambda_{Lag}}e^{(\lambda-\delta)s})~\mbox{and}~ A_s^*=(\frac{h'}{\varphi'})^{-1}(\frac{1}{\lambda_{Lag}}e^{(\lambda-\delta)s}) { \vee 0}\,\mbox{ for all }  s\in [0,\infty),
 \end{equation}
where the Lagrange multiplier $\lambda_{Lag}$ satisfies the following equation:
\begin{equation}\label{cas01}
  \int_0^\infty e^{-{{{\lambda}} s}}\left(U\left((U')^{-1}(\frac{1}{\lambda_{Lag}}e^{(\lambda-\delta)s})\right)
  -h\left((\frac{h'}{\varphi'})^{-1}(\frac{1}{\lambda_{Lag}}e^{(\lambda-\delta)s}){\vee 0}\right)\right)ds=\underline{x}.
\end{equation}
As a function of the time, 
 the optimal rent $R^*$ is  non-increasing  and the optimal effort $A^*$ is non-decreasing.
\end{proposition}
  This means that since the Agent is more impatient than the Principal, he will receive a higher rent  and he will provide a smaller effort  at the beginning of the contract than in the future. \\
 
\begin{proof}
We fix $s\in [0,\infty)$. The function $x\longrightarrow -e^{-\delta s}x +\lambda_{Lag} e^{-\lambda s} U(x)$ is concave on $[0,\infty)$.
The first order condition for maximization on $(0,\infty)$ is given by:
\begin{eqnarray*}
-e^{-\delta s} +\lambda_{Lag} e^{-\lambda s} U^{'}(x^*)=0.
\end{eqnarray*}
Here $\lambda_{Lag}$ must be positive, otherwise $e^{-\delta s}=0$ which is false. We deduce that the optimal rent  is positive, deterministic and is given by:
\begin{eqnarray*}
R_s^*=(U')^{-1}(\frac{1}{\lambda_{Lag}}e^{(\lambda-\delta)s}).
\end{eqnarray*}
For the optimal effort, we consider on $(0,\infty)$ the function $x\longrightarrow -e^{-\delta s}\varphi(x) -\lambda_{Lag} e^{-\lambda s} h(x)$, which is concave.
  The first order condition of optimality  is given by
\begin{eqnarray*}
-e^{-\delta s} \varphi'(x^*) -\lambda_{Lag} e^{-\lambda s} h^{'}(x^*)=0,  \mbox{ if } x^*>0
\end{eqnarray*}  
otherwise $x^{*}=0$. 
We deduce that the optimal effort  is non-negative, deterministic and is given by:
\begin{eqnarray*}
 A_s^*=(\frac{h'}{\varphi'})^{-1}(\frac{1}{\lambda_{Lag}}e^{(\lambda-\delta)s})\vee 0.
\end{eqnarray*}  
The linear independence constraint qualification is satisfied as $U'({R^*_s})$ and $h'({A^*_s})$ are positive.
From Assumption \ref{hyp}, and since $\varphi$ and $U$ are strictly concave, $h$ is strictly convex, then we have
\begin{eqnarray*}
 \left\{
    \begin{array}{ll}
        e^{-\delta t}\varphi^{''}({A^*_s})-\lambda_{Lag}e^{-\lambda t}h^{''}({A^*_s})<0,\\
     \lambda_{Lag}e^{-\lambda t}U^{''}({R^*_s})<0,
    \end{array}
\right. 
\end{eqnarray*}
and so the sufficient conditions of optimality are satisfied.
The optimal rent  and the optimal effort are deterministic controls  given by (\ref{effortrent}).\\
It follows that: 
\begin{eqnarray}\label{contrai}
& &\esp\left[\int_0^\infty e^{-{{{\lambda}} s}}(U(R_s^*)-h(A_s^*))ds\right]\\
  &=&\int_0^\infty e^{-{{{\lambda}} s}}
  \left(
  U\left((U')^{-1}(\frac{1}{\lambda_{Lag}}e^{(\lambda-\delta)s})\right)
  -h\left((\frac{h'}{\varphi'})^{-1}(\frac{1}{\lambda_{Lag}}e^{(\lambda-\delta)s} ) \vee 0\right)
  \right)ds.\nonumber
\end{eqnarray}
The first order condition with respect to $\lambda_{Lag}$ is given by the complementary condition  i.e.
  $$\lambda_{Lag}\Big(\esp\left[\int_0^\infty e^{-\lambda s}(U(R_s^*)-h(A_s^*))ds\right]-\underline{x}\Big)=0.$$
Since $\lambda_{Lag}>0$ and using (\ref{contrai}), we obtain  Equation (\ref{cas01})  satisfied by $\lambda_{lag}$ (the reservation constraint is binded). 
{ The monotonicity of $s\rightarrow R_s^*=(U')^{-1}(\frac{1}{\lambda_{Lag}}e^{(\lambda-\delta)s})$ (non-increasing)  and of $s\rightarrow  A_s^*=(\frac{h'}{\varphi'})^{-1}(\frac{1}{\lambda_{Lag}}e^{(\lambda-\delta)s}) \vee 0$ (non-decreasing) is a direct consequence of the increase  of $s \rightarrow e^{(\lambda-\delta)s}$ (since $\lambda \geq  \delta$), the decrease of $(U')^{-1}$ (since $U$ is concave) and the increase of $(\frac{h'}{\varphi'})^{-1}$ (since $h$ is convex and $\varphi$ is concave). If $\lambda >  \delta$ then $s\rightarrow R_s^*$ is decreasing.  }
\end{proof}\pf
\begin{ex}\label{exnum}
  Suppose that the functions $\varphi$ and $h$ are given by $\varphi(x)=3(1-\exp(-\alpha x))$ and $h(x)=\exp(\beta x)-1.$
  Then $\varphi'(x)=3\alpha\exp(-\alpha x),~h'(x)=\beta\exp(\beta x)$,
  which implies that $\frac{h'}{\varphi'}(x)=\frac{\beta}{3\alpha}\exp((\beta+\alpha)x)$,
  and so  $(\frac{h'}{\varphi'})^{-1}(x)=\frac{1}{\alpha+\beta}\log(\dfrac{3\alpha}{\beta}x)$.
  In the case of power utility function, i.e. $U(x)=cx^p$ {(0<p<1)} and so $(U')^{-1}(x)=(\frac{x}{pc})^{\frac{1}{p-1}}$, (\ref{effortrent}) yields the expressions
$${R_t^*=\frac{1}{(pc\lambda_{Lag})^{\frac{1}{p-1}}}\exp\left(\frac{(\lambda-\delta)t}{p-1}\right)}\,\mbox{ and }A^*_t=\frac{1}{\alpha+\beta}\log\left(\frac{3\alpha}{\beta\lambda_{Lag}}\exp((\lambda-\delta)t)\right){ \vee 0}.$$
 If $\lambda_{Lag}  \leq  3 \frac{\alpha}{ \beta}$, then $A^*_t=\frac{1}{\alpha+\beta}\log\left(\frac{3\alpha}{\beta\lambda_{Lag}}\exp((\lambda-\delta)t)\right)$ for all $t$.
$$
\hspace*{-1cm} \mbox{ If } \lambda_{Lag}  >  3 \frac{\alpha}{ \beta}, \mbox{ then }  
A^*_t= \left\{
    \begin{array}{ll}
    0 & \mbox{ for  } (\lambda-\delta) t \leq  -\log\left(\frac{3\alpha}{\beta\lambda_{Lag}} \right) \\
      \frac{1}{\alpha+\beta}\log\left(\frac{3\alpha}{\beta\lambda_{Lag}}\exp((\lambda-\delta)t)\right) & \mbox{ for } (\lambda-\delta) t >  -\log\left(\frac{3\alpha}{\beta\lambda_{Lag}} \right).
    \end{array}
\right.
$$
The Lagrange multiplier $\lambda_{Lag}$ is solution to equation \eqref{contrai} (which has  two different expressions depending on the interval considered)
$$
 \left\{
    \begin{array}{ll}
    \dfrac{c}{(pc\lambda_{Lag})^{\frac{p}{p-1}}(\lambda-\frac{p}{p-1}(\lambda-\delta))}+\frac{1}{\lambda}- \left(\frac{3\alpha}{\beta\lambda_{Lag}}\right)^{\frac{\beta}{\alpha+\beta}}\dfrac{\alpha+\beta}{\lambda\alpha+\beta\delta}&=\underline x   \mbox{ on   } [0, 3 \frac{\alpha}{ \beta}] \label{lamlag} \\
    \dfrac{c}{(pc\lambda_{Lag})^{\frac{p}{p-1}}(\lambda-\frac{p}{p-1}(\lambda-\delta))}
+ \left( \frac{1}{\lambda}\left(\frac{3\alpha}{\beta \lambda_{Lag}}\right)^{\frac{\lambda}{\lambda-\delta}}
-\dfrac{\alpha+\beta}{\lambda\alpha+\beta\delta} 
\left(\frac{3\alpha}{\beta\lambda_{Lag}}   \right)^{\frac{\lambda}{\lambda-\delta}  } \right) \mathbf 1_{\lambda \neq \delta}& =\underline x
\mbox{ on   }  ]3 \frac{\alpha}{ \beta}, +\infty[. \label{lamlagbis} 
    \end{array}
\right.
$$
\end{ex}

\subsection { The general case}
In the first-best framework, and in the general case, the Principal is faced with the following optimization problem: 
\begin{eqnarray}\label{optimgeneral}
\Sup_{(R,\tau,\xi,A)\in \mathcal{E}}\esp\left[\int_0^\tau e^{-\delta s}(\varphi(A_s)-R_s)ds -e^{-\delta \tau} \xi\right],
\end{eqnarray}
where $\mathcal{E}$ is given by (\ref{ensfb}), and subject to the reservation constraint
$$\esp\left[\int_0^\tau e^{-\lambda s}(U(R_s)-h(A_s))ds +e^{-\lambda \tau} U(\xi)\right]\geq \underline{x}.$$
The main result of Section 3 is the following proposition. It gives the possible cases in the  first-best case. 

\begin{proposition}\label{propositioncas}
Under Assumption \ref{hyp}, in the first-best case, we define  the function $$t\longrightarrow H_t:=\varphi(A_t^*)-R_t^*,$$ where $A^*$ and $R^*$ are given in (\ref{effortrent}). Then
\begin{enumerate}
\item If {$\int_0^{\infty}e^{-\delta s}H_sds>0$}, it is not optimal for the Principal to end the contract ($\tau^*=\infty$ is optimal).
\item If {$\int_0^{\infty}e^{-\delta s}H_sds \leq 0$}, the Principal does not offer the contract to the Agent ({$\tau^*=0$ is optimal}).
\end{enumerate}
\end{proposition}

\begin{proof}
We introduce the Lagrangian problem. The first order conditions inside the expectation with respect to the effort and the rent are the same as in the infinite horizon contract. 
 We know that $\lambda_{Lag}\geq 0$. If $\lambda_{Lag}= 0$, then as in the proof of the infinite horizon case, we deduce that for all $t\geq 0$, we have $R^*_t=0 $, $A^*_t=0 $ (and $H(t)=0$ for all $t$). It is then clear that we must have $\tau=0$ otherwise the reservation constraint is never satisfied .\\
For $\lambda_{Lag}>0$, the function $t\longrightarrow H_t$ is well-defined.\\
 Since $\lambda\geq \delta$,  $t\mapsto \frac{1}{\lambda_{Lag}}e^{(\lambda-\delta)t}$ is non-decreasing.
The strict concavity  of the function $\varphi$ and the strict convexity of the function $h$ yields the function $t\rightarrow (\frac{h'}{\varphi'})^{-1}(\frac{1}{\lambda_{Lag}}e^{(\lambda-\delta)t})$ is non-decreasing. The strict concavity of $U$ and the decrease of $U'$ yields that the function 
$t\rightarrow (U')^{-1}(\frac{1}{\lambda_{Lag}}e^{(\lambda-\delta)t})$ is non-increasing. Therefore the optimal rent is non-increasing and the optimal effort is non-decreasing in time. By monotonicity  of the function $\varphi$, the function $t\rightarrow H_t$ is non-decreasing.\\
We have the following cases:\\
{\underline {First case:}} If $H_0>0$, and since $U^{-1}$ in non-negative, as in Proposition \ref{horinfini}, {it} is optimal for the Principal to never stop the contract ({$\tau=\infty$ is optimal}).\\
{\underline {Second case:}} If $H_0\leq 0$, and if {there} exists a $t_0$ such that $H_{t_0}=0,$
we compute: $\int_0^{t_0} e^{-\delta s}H_sds$ and $\int_{t_0}^\infty e^{-\delta s}H_sds.$
 \vspace*{-3mm}
 \begin{enumerate}
 \item If $|\int_0^{t_0} e^{-\delta s}H_sds|\leq \int_{t_0}^\infty e^{-\delta s}H_sds$, as in Proposition \ref{horinfini}, it is optimal for the Principal  to never stop the contract ({$\tau=\infty$ is optimal}) because the profit is greater than the loss.
 \item If $|\int_0^{t_0} e^{-\delta s}H_sds|\geq \int_{t_0}^\infty e^{-\delta s}H_sds$, it is optimal for  the Principal  not to offer the contract to the Agent ({$\tau=0$ is optimal}) because she knows that the loss will be greater than the profit. \pf
 \end{enumerate}
\end{proof}

\noindent Figure \ref{H_tcas1}  illustrates the monotonicity of $H$  with respect to $t$ obtained in Proposition \ref{propositioncas},  for the numerical case provided in Example \ref{exnum}  with  $\alpha= \beta=0.1$ and $p=\frac{1}{4}$.

    \begin{figure}[H]
        \includegraphics[width=14cm,height=10cm]{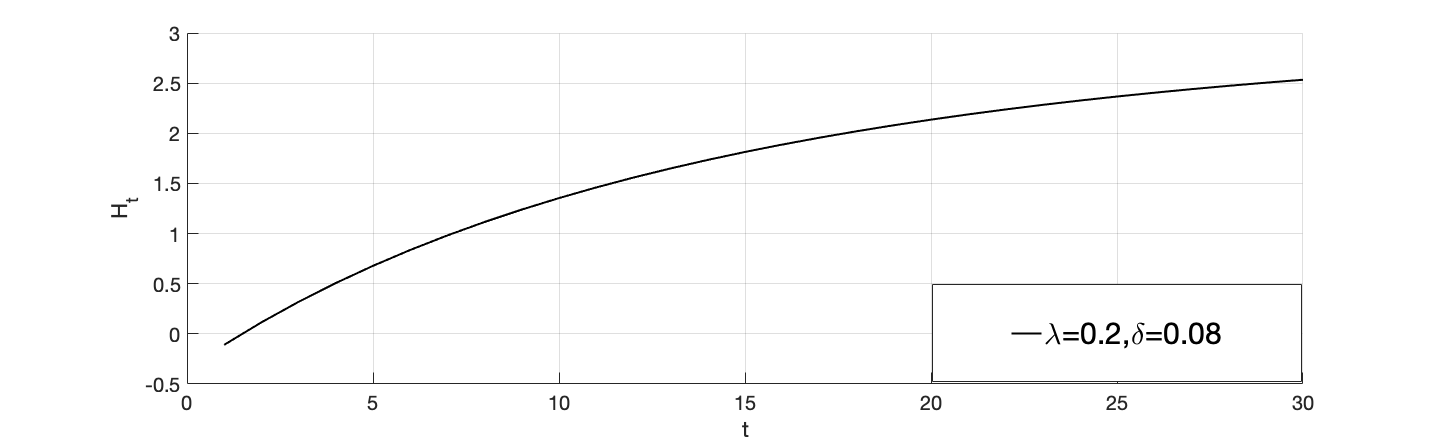}
\caption{$t\mapsto H_t(0)$  }\label{H_tcas1}
    \end{figure}


\subsection{The monotonicity of the controls with respect to the Agent's value function}
\noindent The following proposition shows the monotonicity of the optimal rent and the optimal effort with respect to the Agent's initial value function, denoted $x$.
\begin{proposition}\label{monoto}
 Under  Assumptions \ref{hyp}, in the non-degenerate case  $\tau^* \neq 0$  \\
 (1) The optimal rent is increasing  and the optimal effort is decreasing with respect to $x.$\\
 (2) The function $H_t=\varphi(A_t^*)-R_t^*$ is decreasing with respect to $x$.
\end{proposition} 
To prove this proposition,  we need the following lemma.
\begin{lemme}\label{mono}Under  Assumption \ref{hyp},  
the function $$G:\lambda_{Lag}\mapsto\int_0^\infty e^{-{{{\lambda}} s}}\left(U\left((U')^{-1}(\frac{1}{\lambda_{Lag}}e^{(\lambda-\delta)s})\right)-h\left((\frac{h'}{\varphi'})^{-1}(\frac{1}{\lambda_{Lag}}e^{(\lambda-\delta)s})\right) \vee 0\right)ds.$$
is increasing with respect to $\lambda_{Lag}$ on $(0,\infty)$.
\end{lemme}
\begin{proof}
Since $\tau^* \neq 0$, $\lambda_{Lag}>0$, and we have $\lambda_{Lag}\mapsto \frac{1}{\lambda_{Lag}}e^{-(\lambda-\delta)s}$ is decreasing. The decrease of $U'$ yields that the function $\lambda_{Lag}\mapsto (U')^{-1} (\frac{1}{\lambda_{Lag}}e^{-(\lambda-\delta)s})$ is increasing. As $U$ is increasing, then, we have $\lambda_{Lag}\mapsto U\left((U')^{-1} (\frac{1}{\lambda_{Lag}}e^{-(\lambda-\delta)s})\right)$ is increasing. Using the properties of strict concavity of $\varphi$ and strict convexity of $h$, we have $(\frac{h'}{\varphi'})'>0$. As the function h is increasing, then $\lambda_{Lag}\mapsto h\left((\frac{h'}{\varphi'})^{-1}(\frac{1}{\lambda_{Lag}}e^{(\lambda-\delta)s}) \vee 0 \right)$ is non-increasing. We conclude that  $G$ is increasing with respect to $\lambda_{Lag}.$
\end{proof}\pf\\

\begin{proof}of Proposition  \ref{monoto}.\\
 It is clear that $\lambda_{Lag}$ depends on $x$  since it is solution of the equation $G(\lambda_{Lag})=x$. 
  We consider $x>x^{'}>0$. In this case, if $\lambda_{lag}(x)\leq \lambda_{lag}(x^{'})$, then  $G(\lambda_{lag}(x))\leq G(\lambda_{lag}(x^{'}))$ by Lemma \ref{mono}, which contradicts  that $x > x^{'}$. So we must have $\lambda_{lag}(x)>\lambda_{lag}(x^{'})$,
  and the function  $x\rightarrow \lambda_{Lag}(x)$ is increasing.
For a fixed $t$, we have $x\rightarrow \frac{1}{\lambda_{Lag}(x)}e^{(\lambda-\delta)t}$ is decreasing, and by { using the properties of concavity of $U$ and $\varphi$  and convexity of $h$, we get that  the rent is increasing and the optimal effort is decreasing with respect to $x$}.
 As $A^*$ and $R^*$ depend on $x$ through $\lambda_{Lag}$,  $x$ also impacts $H_t(x)=\varphi(A^*_t(x))-R^*_t(x)$. Therefore  we deduce that the function $H_t$ is decreasing with respect to $x$.
\end{proof}\pf\\

\noindent Figure \ref{figureH1}  illustrates  the monotonicity of $H$  with respect to $x$ obtained in Proposition \ref{monoto},  for the numerical example provided in Example \ref{exnum} with $\alpha= \beta=0.1$ and $p=\frac{1}{4}$.

    \begin{figure}[H]
        \includegraphics[width=14cm,height=10cm]{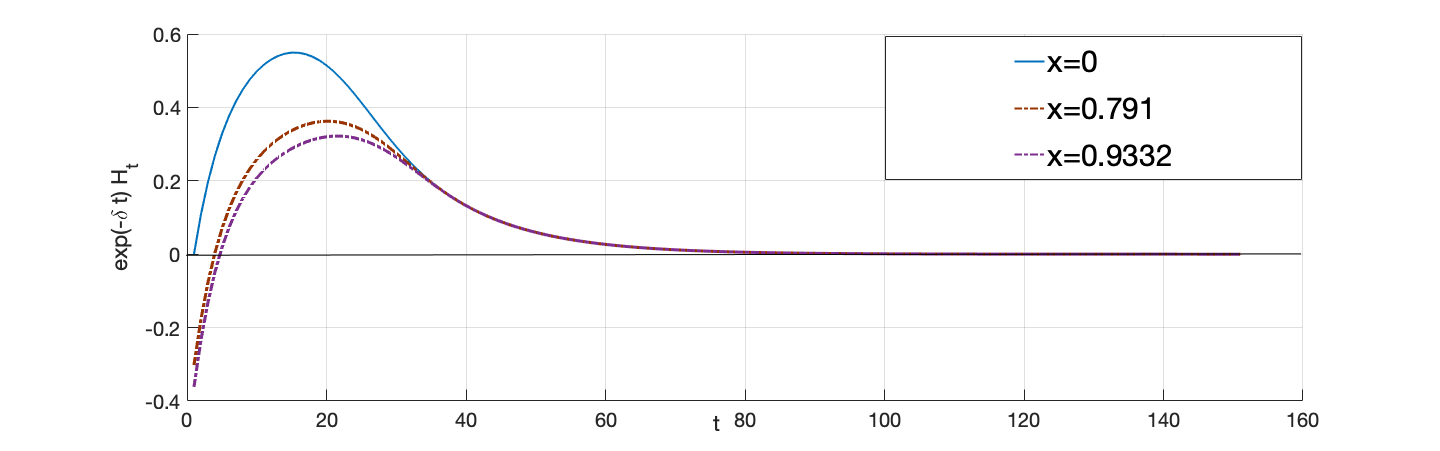}
       \caption{ $t\mapsto e^{-\delta t}H_t(x)$ for different $x$} \label{figureH1}
    \end{figure}

\section{ The second-best case}
In this section, we consider the  second-best case. In this situation of moral hazard, the Principal cannot observe the action chosen by the Agent, and can only control the
rent that she offers, the decision to stop the contract and the cost that follows. The trick to overcome this difficulty is to reformulate the optimization problems in terms of the Agent objective function $J^{Ag}(\Gamma,A)$. The methodology consists on deriving  the dynamics of the objective function of the Agent and
characterizing the incentive compatible contracts.
Then, we enlarge the set of admissible controls for the Principal by choosing them $\ff$-progressively measurable.
We solve the optimization problem of the Principal, which is a standard mixed optimal stopping/stochastic control problem on the large set of admissible controls.
Finally, we prove  that  at the optimum the optimal contract is  $\ff^X$-progressively measurable and so we solve the initial optimization problem of the Principal.

\subsection{The incentive compatible contract for the Agent}\label{section3-4}
For a fixed stopping time $\tau\in \mathcal{T}$, we introduce the following spaces  which are useful to characterize the structure of the incentive contract.
\begin{eqnarray*}
\mathcal{S}_\rho^2(\tau):&=&\{  {Y}~ \R\mbox{-valued,}~\ff\mbox{-progressively measurable continuous process such that}\\
& &
~~~~~~~~~~~~||{Y}||_{{\cal S}_\rho^2(\tau)} := \left(\esp^\pr \left[\Sup_{ 0\leq s\leq\tau}e^{(\rho-2\lambda) s} |{Y}_s|^2\right]\right)^{\frac{1}{2}}<\infty\},\\
{\cal H}_\rho^2(\tau):&=&\{{Z}~ \R\mbox{-valued,}~\ff\mbox{-progressively measurable process such that}\\
& &
~~~~~~~~~~~~||{Z}||_{{\cal H}_\rho^2(\tau)} := \left(\esp^\pr\left[ |\int_0^{\tau} e^{(\rho-2\lambda) s} |Z_s|^2ds|\right]\right)^{\frac{1}{2}}<\infty\}.
\end{eqnarray*}
The following lemma gives the dynamics of the Agent's objective function. 
\begin{lemme}\label{lemmeincentive} Suppose Assumption \ref{hyp}.
For any admissible contract and effort $(\Gamma, A)\in\mathcal{D}^P_\rho\times \mathcal{D}^{Ag}_\rho$,  there exists  $Z^A\in {\cal H}_\rho^2(\tau)$ such that the dynamics of the Agent's objective function $J_t^{Ag}(\Gamma,A)$ evolves according to  BSDE \eqref{Ii01} with random terminal condition
\begin{equation}\label{Ii01}
dJ_t^{Ag}(\Gamma,A)=-\left[-{{\lambda}} J_t^{Ag}(\Gamma,A)+U(R_t)+\psi(A_t,Z_t ^A)\right]dt+Z_t^A dX_t,~J_\tau^{Ag}(\Gamma,A)=U(\xi)\un_{\{\tau<\infty\}}
\end{equation}
where 
\begin{equation}\label{psii}
\psi(a,z):=-h(a)+{z}\frac{\varphi(a)}{\sigma}.
\end{equation}
\end{lemme}
\begin{proof}
  We fix $(\Gamma, A)\in\mathcal{D}^P_\rho\times \mathcal{D}^{Ag}_\rho$.
  We define the martingale $(M_t)_{0\leq t\leq \tau}$ by
\begin{eqnarray}
  M_t:=\esp \left[\int_0^\tau e^{-{{{\lambda}}s}}(U(R_s)-h(A_s))ds+e^{-{{\lambda}}\tau}U(\xi){\un_{\{\tau<\infty\}}}|\F_t\right].
\end{eqnarray}
As $(M_t)_{0\leq t\leq \tau}$ is square integrable, then by the martingale representation theorem, there exists $(\hat Z^A_t)_{0\leq t\leq \tau}$ $\ff$-adapted
process such that $\esp \left[\int_0^\tau| \hat Z^A_s|^2ds\right]<\infty$ and $M_t=M_0+\int_0^t Z_s^A dW_s$, which implies  for $t \in \lbr 0, \tau \lbr$ 
\begin{eqnarray}
 e^{-\lambda t}J_t^{Ag}(\Gamma,A) =J_0^{Ag}(\Gamma,A)+\int_0^t  e^{-{{{\lambda}}s}}(U(R_s)-h(A_s))ds+\int_0^t {\hat Z}_s^A dW_s.
\end{eqnarray}
We put ${ Z}_t^A=e^{\lambda t}{\hat Z}_t^A$ for all $0\leq t\leq \tau$. From Proposition 3.1 in Hajjej et al. \cite{hajjej2019optimal}, 
$(J_t^{Ag}(\Gamma,A),{ Z}_t^A)_{0\leq t\leq \tau}\in \mathcal{S}_\rho^2(\tau)\times{\cal H}_\rho^2(\tau) $.
By using Cauchy Schwarz inequality, we have 
\begin{eqnarray}
\int_0^t|{\hat Z}_s^A\frac{\varphi(A_s)}{\sigma}| ds &\leq & \sqrt{\int_0^te^{\rho s}|{\hat Z}_s^A|^2ds} \sqrt{ \int_0^t e^{-\rho s}|\frac{\varphi(A_s)}{\sigma}|^2 ds}\\
&\leq&  C \sqrt{\int_0^te^{\rho s}|{\hat Z}_s^A|^2ds} <\infty,\nonumber
\end{eqnarray}
where the second inequality is obtained from the boundedness of $\varphi$ (see  Assumption \ref{hyp}) and $C$ is a positive constant. 
As the integral $\int_0^t{\hat Z}_s^A\frac{\varphi(A_s)}{\sigma} ds$ is well-defined, and from the definition of the output process (\ref{Sv}), we have
$\int_0^t {\hat Z}_s^A dW_s= \int_0^t \frac{{\hat Z}_s^A}{\sigma} dX_s-\int_0^t\frac{{\hat Z^A}_s\varphi(A_s)}{\sigma} ds$, leading to 
 BSDE(\ref{Ii01}). 
\end{proof}
\begin{remarque}
The integral with respect to $X$ appears naturally in this second-best case. It is well-defined for all $\Gamma\in \mathcal{D}^P_\rho$ and  $A\in \mathcal{D}^{Ag}_\rho$. It has a sense for the Principal since she observes only the output process and could not make the difference between the impact of the effort and the  Brownian motion. 
\end{remarque}
The associated Backward Stochastic Differential Equation (\ref {Ii01}) can be written as
\begin{eqnarray}\label{BSDE-3}
 \left\{
    \begin{array}{ll}
       dY_t &=-\left(-{\lambda}Y_t+U(R_t)+\psi(A_t,Z_t^A)\right)dt+ Z_t^AdX_t \\
        Y_\tau &=U(\xi)\un_{\{\tau<\infty\}}
    \end{array}
\right. 
\end{eqnarray}
Considering the discounted quantities 
 $$(\tilde{Y}_t,\tilde{Z}_t^A)=(e^{-\lambda t}Y_t,e^{-\lambda t} Z_t^A),~ ~dt\otimes d\pr ~a.e,\,t\in\lbr 0,\tau \lbr,$$
 allows us to get rid of the $Y$-term in the drift of BSDE \eqref{BSDE-3}: 
$(\tilde{Y},\tilde{Z}^A)$ satisfies  BSDE \eqref{BSDE2}
\begin{eqnarray}\label{BSDE2}
 \left\{
    \begin{array}{ll}
       d\tilde{Y}_t &=-\left(\tilde{U}(R_t)+\tilde{\psi}(A_t,\tilde{Z}_t^A)\right)dt+\tilde{Z}^A_tdX_t, \\
        \tilde{Y}_\tau &={  \tilde U(\xi)}{\un_{\{\tau<\infty\}}},
    \end{array}
\right. 
\end{eqnarray}
where \vspace{-0.5cm}
\begin{eqnarray}\label{psiU}
 \left\{
    \begin{array}{ll}
           {\tilde{h}(A_t)}&:=e^{-\lambda t}h(A_t),\\
                   \tilde{U}(R_t) &:=e^{-\lambda t}U(R_t),\\ \tilde{U}(\xi) &:=e^{-\lambda t}U(\xi)\\
      \tilde{\psi}(A_t,{\tilde{ Z}}_t^A) &:=-{\tilde{h}(A_t)}+ \tilde{Z}_t^A\frac{\varphi(A_t)}{\sigma}.\\
    \end{array}
\right. 
\end{eqnarray}
The next result is a comparison theorem for BSDE (\ref{BSDE-3}), where the stochastic integral is defined with respect to the semimartingale $X$. This is the difference with the comparison theorem proved in Hajjej et al. \cite{hajjej2019optimal}.
\begin{proposition}\label{rxisunicom}
   Suppose Assumption \ref{hyp}.\\
 1. There exists a unique $(Y, Z^A)\in \mathcal{S}_\rho^2(\tau)\times{\cal H}_\rho^2(\tau)$ solving the BSDE (\ref{BSDE-3}).\\
 2.  {Let $\Gamma=(R,\tau,\xi)\in \mathcal{D}^P_\rho$} and  $A^i\in \mathcal{D}^{Ag}_\rho$ for $i=1,2$.
 Let $({Y}^i,{Z}^i)\in \mathcal{S}_\rho^2(\tau)\times{\cal H}_\rho^2(\tau)$ be the solution of the following BSDE
\begin{eqnarray}\label{BSDE22}
 \left\{
    \begin{array}{ll}
       d{Y}_t^i &=-\left(-\lambda Y_t+{U}(R_t)+{\psi}(A_t^i,{Z}_t^i)\right)dt+{Z}^i_tdX_t, \\
        {Y}_\tau^i &=U(\xi){\un_{\{\tau<\infty\}}}.
    \end{array}
\right. 
\end{eqnarray}
{\begin{equation}\label{comp1}
\mbox{ If }{\psi}(A_t^1,{Z}_t^2)\leq {\psi}(A_t^2,{Z}_t^2)~dt\otimes d\mathbb{P}\,a.e.,~\mbox{for all}~t\in\lbr 0,\tau \lbr,
\end{equation}}
then
$${Y}_t^1\leq {Y}_t^2~dt\otimes d\mathbb{P}\,a.e.,~\mbox{for all}~t\in\lbr 0,\tau \lbr.$$
\end{proposition}
\begin{proof}
 1. The existence is immediate since $Y_t=J_t^{Ag}(\Gamma,A)$ for all $0\leq t\leq \tau$ and the process $Z^A$ appears naturally
 by applying the martingale representation theorem.
 As $\Gamma=(R,\tau,\xi)\in \mathcal{D}^P_\rho$ and  $A\in \mathcal{D}^{Ag}_\rho$, then $Y\in \mathcal{S}^2_\rho(\tau)$ and $Z^A\in {\cal H}^2_\rho(\tau)$.
 For the unicity, we fix $0\leq t\leq \tau$. We assume that there exists two solutions $(Y^1, Z^{1,A})$ and $(Y^2, Z^{2,A})$ to the BSDE (\ref{BSDE-3}), then
  \begin{eqnarray}
    {\tilde Y}^1_t-{\tilde Y}^2_t&=&\int_t^\tau\frac{\varphi(A_s)({\tilde Z}^{1,A}_s-{\tilde Z}^{2,A}_s)}{\sigma}ds-\int_t^\tau({\tilde Z}^{1,A}_s-{\tilde Z}^{2,A}_s)dX_s\\
    &=&\int_t^\tau(Z^{1,A}_s-Z^{2,A}_s)dW_s.
  \end{eqnarray}
  Taking the conditional expectation, we obtain $Y_t^1=Y_t^2$  $dt\otimes d\pr$ a.e. for all $0\leq t\leq \tau$, and so
  $Z_t^1=Z_t^2$ $dt\otimes d\pr$ a.e. for all $0\leq t\leq \tau$.\\
  2. From (\ref{BSDE22}), we have
\begin{eqnarray*}
\tilde{Y}^1_t-\tilde{Y}^2_t&=&\int_t^{\tau}\left(\tilde{\psi}(A^1_s,\tilde{Z}_s^{1})-\tilde{\psi}(A^2_s,\tilde{Z}_s^{2})\right)ds-\int_t^{\tau}\left(\tilde{Z}_s^1-\tilde{Z}_s^2\right)dX_s+\tilde{Y}_{\tau}^1-\tilde{Y}_{\tau}^2\\
&=&\int_t^{\tau}\left(\tilde{\psi}(A^1_s,\tilde{Z}_s^{1})-\tilde{\psi}(A^2_s,\tilde{Z}_s^{2})+\tilde{\psi}(A^1_s,\tilde{Z}_s^{2})-\tilde{\psi}(A^1_s,\tilde{Z}_s^{2})\right)ds
-\int_t^{\tau}\left(\tilde{Z}_s^1-\tilde{Z}_s^2\right)dX_s\\
&\leq& \int_t^{\tau}\left(\tilde{\psi}(A_s^1,\tilde{Z}_s^1)-\tilde{\psi}(A_s^1,\tilde{Z}_s^2)\right)ds
-\int_t^{\tau}\left(\tilde{Z}_s^1-\tilde{Z}_s^2\right)dX_s,
\end{eqnarray*}
where the last inequality is obtained by using inequalities (\ref{comp1}) and (\ref{psiU}). We obtain
\begin{eqnarray*}
\tilde{Y}^1_t-\tilde{Y}^2_t&\leq& \int_t^{\tau}(\tilde{Z}_s^1-\tilde{Z}_s^2)\frac{\varphi(A_s^1)}{\sigma}ds
-\int_t^{\tau}\left(\tilde{Z}_s^1-\tilde{Z}_s^2\right)dX_s\\
&=&-\int_t^{\tau}\left(\tilde{Z}_s^1-\tilde{Z}_s^2\right)dW_s.
\end{eqnarray*}
By taking the {conditional} expectation under $\pr$, the stochastic integral $\int_t^{\tau}(\tilde{Z}_s^1-\tilde{Z}_s^2)dW_s$ vanishes. {As ${Y}^i$ is $\mathbb{F}$-progressively measurable process,}   we obtain $${Y}_t^1\leq {Y}_t^2 \,dt\otimes d\pr\,a.e.,\,~t\in\lbr 0,\tau\lbr.$$ 
\end{proof}
\pf \\
The next lemma is useful to parametrize the optimal effort as a function of the process $Z$.
\begin{lemme}\label{lemme01-3} [Lemma 3.5 in \cite{hajjej2019optimal}]
Suppose  Assumption \ref{hyp}. Let $z$ be a real number and  define $A^*(z):=\arg\Max_{a\geq0}\psi(a,z).$ If $z>\sigma\frac{h'(0)}{\varphi'(0)},$ then $A^*(z)=(\frac{h'}{\varphi'})^{-1}(\frac{z}{\sigma})$ and if $z\leq\sigma\frac{h'(0)}{\varphi'(0)},$ then $A^*(z)=0.$
{Moreover, $A^*$ is  a bijection from $\{0\} \cup(\sigma\frac{h'}{\varphi'}(0),\infty)$ to $[0,\infty).$ }
\end{lemme}
The following proposition gives the  bijection between the process  $(Z_t^A)_{t\geq0}$ and the candidate for optimal effort $ (A^*_t)_{t\geq 0}.$ 
 \begin{proposition}\label{exitenceastar} [ Proposition 3.6 in \cite{hajjej2019optimal}]
Let $A\in\mathcal{D}^{Ag}_\rho.$ There exists a bijection between the process  $(Z_t^A)_{t\geq0}$ and the optimal effort $ (A^*_t)_{t\geq 0}.$ The bijection is given by
      $$A^*_t=A^*(Z_t^A)=(\frac{h'}{\varphi'})^{-1}(\frac{Z_t^A}{\sigma})\un_{\{Z_t^A>0\}},$$
      {or equivalently}
      $$ Z_t^A= (\sigma\frac{h'}{\varphi'})(A^*(Z_t^A))\un_{\{A^*(Z_t^A)>0\}}.$$          
  \end{proposition}
 The following proposition shows the structure of the incentive compatible contracts in $\mathcal{D}^{Ag}_\rho$.
\begin{proposition}\label{incentive}
Let  $\Gamma=(R,\tau,\xi)\in \mathcal{D}^P_\rho$ and $Z\in {\cal H}^2_\rho(\tau)$. If $A^*(Z)\in \mathcal{D}^{Ag}_\rho$, then for all $A\in \mathcal{D}^{Ag}_\rho$ such that
  \begin{eqnarray}\label{inecomp}
\psi(A_t,Z_t)\leq \psi(A^*(Z_t),Z_t), t\in\lbr 0,\tau\lbr,\,dt\otimes d\pr a.e.,
  \end{eqnarray}  
  we have $J_t^{Ag}(\Gamma,A)\leq J_t^{Ag}(\Gamma,A^*(Z))$, $t\in\lbr 0,\tau\lbr,\,dt\otimes d\pr$ a.e.
\end{proposition}
\begin{proof}
  Let $\Gamma=(R,\tau,\xi)\in \mathcal{D}^P_\rho$ and $Z\in {\cal H}_\rho^2(\tau)$. As $A^*(Z)\in \mathcal{D}^{Ag}_\rho$,
  then $(J_t^{Ag}(\Gamma,A^*(Z)))_{0\leq t\leq \tau }$ solves the BSDE 
  \begin{eqnarray*}\label{Ii1-3}
  dJ_t^{Ag}(\Gamma,A^*(Z))&=&-\left(-{{\lambda}} J_t^{Ag}(\Gamma,A^*(Z_t))+U(R_t)+\psi(A^*(Z_t),Z_t)\right)dt+Z_t dX_t,\nonumber\\
  J_\tau^{Ag}(\Gamma,A^*(Z))&=&U(\xi)\un_{\{\tau<\infty\}}.
  \end{eqnarray*}
  From inequality (\ref{inecomp}) and using the comparison theorem (see Proposition \ref{rxisunicom}), we conclude that
  $J_t^{Ag}(\Gamma,A)\leq J_t^{Ag}(\Gamma,A^*(Z)), t\in\lbr 0,\tau\lbr,\,dt\otimes d\pr$ a.e.
\end{proof}\pf

\subsection{Hamilton Jacobi Bellman Variational Inequality}
The parametrization of the optimal effort as a deterministic function of a process $Z$ is useful for the Principal to solve her optimization problem   
which can be written as a stochastic control problem  under partial information. We adopt a forward point of view for the dynamics of the Agent's objective function
which evolves according to  the following forward SDE:
{\small
\begin{eqnarray}\label{eqJforwar}
  dJ_t^{Ag}(x,R,\tau,A^*(Z))&=&\left(\lambda J_t^{Ag}(x,R,\tau,A^*(Z))-U(R_t)+h(A^*(Z_t))-Z_t\frac{\varphi(A^*(Z_t))}{\sigma}\right)dt+Z_tdX_t, \nonumber\\
  J_0^{Ag}(x,R,\tau,A^*(Z))&=&x\geq \underline{x}.
\end{eqnarray}}
  SDE (\ref{eqJforwar}) is well-defined for $(\Gamma, A^*(Z))\in \mathcal{D}^P_\rho\times \mathcal{D}^{Ag}_\rho$.
We recall that the initial condition of the SDE (\ref{eqJforward-3}) satisfies 
  the reservation constraint formulated in the maximization problem of the Principal.
  However, we solve the stochastic control problem related to the Principal on the whole domain i.e on $\R^+$.
  In fact the Agent's objective function at time $t$, denoted by $J_t^C(x,R,\tau,A^*(Z))$ could be less than $ \underline x$ although $J_0^C(x,R,\tau,A^*(Z))\geq  \underline x$. 
The process $(J_t^{Ag}(x,R,\tau,A^*(Z))_{t\geq 0}$ is considered as a state variable, and the contract $\Gamma$ and the optimal effort $A^*(Z)$ as  control processes
which must be $\ff^X$- adapted for the principal.
From Proposition \ref{exitenceastar}, there exists a bijection between  $Z$ and $A^*(Z)$ and $ Z_t=(\sigma\frac{h'}{\varphi'})(A^*(Z_t))\un_{\{A^*(Z_t)>0\}}$ for all $t\geq 0$.
The Principal's value function  at time $0$  is related to a stochastic control under partial information, and it is defined as follows: 
\begin{equation} \label{valuefunctiono}
v(x) :=\Sup_{(R,\tau,{{A^*(Z)}})\in {\cal{G}^X}
}\esp\left[\int_0^\tau e^{-\delta s}(\varphi(A^*_s(Z))-R_s)ds-e^{-\delta \tau}U^{-1}(J_\tau^{Ag}(x,R,\tau,A^*(Z)))\right],
\end{equation}
where  ${\cal{G}^X}$ is given by 
{\begin{eqnarray*}
{\cal{G}^X} &:=& \{(R,\tau,A^*(Z))~ R\geq 0 ~{\ff}^X\mbox{-progressively measurable},\tau \in {\cal{T}^X}, { A^*(Z)}\geq0~{\ff}^X\mbox{-progressively}\\
& &
\mbox{ measurable such that}~\esp \left[\int_0^\infty e^{(\rho-2\delta) s} |\varphi(A^*(Z_s))|^2ds\right]<\infty,\,\, \esp \left[\int_0^\infty e^{(\rho-2\delta) s}|R_s|^2ds\right]<\infty, \\
& &\mbox{ and } \esp \left[e^{(\rho-2\delta) s}|U^{-1}(J_\tau^{Ag}(x,R,\tau,A^*(Z)))|^2\un_{\{\tau<\infty\}}\right]<\infty\}.
\end{eqnarray*}}
\noindent As $W$ is not a $\ff^X$-Brownian motion, it is not clear that the dynamic programming principle holds for $v$ and it is not immediate to derive the associated dynamic programming equation. 
To overcome this difficulty, we introduce an auxiliary Principal's value function at time $0$, when the controls are $\ff$-adapted. It is defined by: 
\begin{equation} \label{valuefunction}
v^{aux}(x) :=\Sup_{(R,\tau,{{A^*(Z)}})\in {\cal{G}}
}\esp\left[\int_0^\tau e^{-\delta s}(\varphi(A^*_s(Z))-R_s)ds-e^{-\delta \tau}U^{-1}(J_\tau^{Ag}(x,R,\tau,A^*(Z)))\right],
\end{equation}
where ${\cal{G}}$ is given by 
{\begin{eqnarray*}
{\cal{G}} &:=& \{(R,\tau,A^*(Z))~ R\geq 0 ~{\ff}\mbox{-progressively measurable},\tau \in {\cal{T}}, { A^*(Z)}\geq0~{\ff}\mbox{-progressively}\\
& &
\mbox{ measurable such that}~\esp \left[\int_0^\infty e^{(\rho-2\delta) s}| \varphi(A^*(Z_s))|^2ds\right]<\infty,\,\, \esp \left[\int_0^\infty e^{(\rho-2\delta) s}|R_s|^2ds\right]<\infty, \\
& &\mbox{ and } \esp \left[e^{(\rho-2\delta) s}|U^{-1}(J_\tau^{Ag}(x,R,\tau,A^*(Z)))|^2\un_{\{\tau<\infty\}}\right]<\infty\}.
\end{eqnarray*}}
As all the controls are $\ff$-adapted, it is more convenient to use the following structure of the SDE (\ref{eqJforwar}), where the stochastic integral is driven by the $\ff$-Brownian motion:
\begin{eqnarray}\label{eqJforward-3}
  dJ_t^{Ag}(x,R,\tau,A^*(Z))&=&\left({{\lambda}} J_t^{Ag}(x,R,\tau,A^*(Z))-U(R_t)+h(A^*(Z_t))\right)dt\\
  &+&(\sigma\frac{h'}{\varphi'})(A^*(Z_t))\un_{\{A^*(Z_t)>0\}}dW_t.\nonumber
\end{eqnarray}
The Hamilton Jacobi Bellman Variational Inequality (HJBVI) associated to the auxiliary value function is given by:
\begin{equation}\label{IVHJB-3}
\min\left\{\delta w(x)-\Sup_{{(r,a)\in \R^{+}\times \R^{+}}}[{\cal L}^{a,r} w(x)+\varphi(a)-r],w(x)+U^{-1}(x)\right\}=0,\,\, x\in(0,\infty),
\end{equation}
where the second order differential operator ${\cal L}^{a,r}$ is defined by
$${\cal L}^{a,r}w(x):=\frac{1}{2}{(\sigma\frac{h'(a)}{\varphi'(a)})^2}\un_{a>0}w"(x)+[{\lambda} x-U(r)+h(a)]w'(x).$$
Lemma \ref{lemme11} gives the boundary condition $v(0)$ and the growth property satisfied by the value function $v$.
These results will be useful for the verification theorem. The proof is similar as in Hajjej et al. \cite{hajjej2019optimal} and thus is omitted.

\begin{lemme}\label{lemme11}
(1) The value function  $v^{aux}$ defined in (\ref{valuefunction}) satisfies
\begin{equation}\label{CB}
v^{aux}(0)=0.
\end{equation}
(2) There exists a positive constant $K$ such that  
\begin{equation}\label{eqcroissance}
\mbox{for all}\quad  x \geq 0, |v^{aux}(x)| \leq K+U^{-1}(x).
\end{equation}

\end{lemme}

The following result is a verification theorem, stating  that if there exists a smooth solution to the
(HJBVI) (\ref{IVHJB-3}), then it coincides with the value function (\ref{valuefunction}).
{As we do not make a change of probability in the definition of the Principal's value function, the verification theorem requires weaker integrability conditions than in  \cite{hajjej2019optimal}.}
{\begin{theoreme}\label{verification} [Verification Theorem] \\ 
 {We suppose that there exists   a constant $\hat b > 0$ and  a continuous function $w: \mathbb R^+ \longrightarrow \R$ s.t.: \\
  (i) w(0)=0, $w\in C^2([0,\hat b))$ satisfying the  growth condition \eqref{eqcroissance},\\
  (ii) $w>-U^{-1}$ on $(0,\hat b)$ and  $w=-U^{-1}$ on $[\hat b,\infty)$\\
  (iii) $\delta w(x)-\Sup_{(r,a)\in\R^+\times\R^+}\{{\cal{L}}^{a,r}w(x)+\varphi(a)-r\}=0$ for all $x \in (0,\hat b)$.\\
(iv) $\delta (-U^{-1}(x))-\Sup_{(r,a)\in\R^+\times\R^+}\{{\cal{L}}^{a,r}(-U^{-1}(x))+\varphi(a)-r\} \geq 0$ for all $x \in  [\hat b,\infty)$.\\ }
\noindent We also assume that 
\begin{equation}\label{hyp3}
{{\Sup_{(R,\tau,A^*(Z))\in{\cal{G}} }\esp \left[|e^{-\delta\tau}{   U^{-1}}(J_\tau^{Ag}(x,R,\tau,A^*(Z)))|\right]<\infty},}
\end{equation} 
Then we have:
\begin{itemize}
\item [(1)]   $w(x)\geq v^{aux}(x)$ for  any { $x\geq 0$},
\item [(2)] Suppose there exists two  measurable non-negative functions $(a^*,r^*)$ defined on $\mathbb R^+$ s.t. 
\begin{eqnarray}\label{conop} 
  \Sup_{(r,a)\in\R^+\times\R^+}\{{\cal{L}}^{a,r}w(x)+\varphi(a)-r\}={\cal{L}}^{a^*(x),r^*(x)}w(x)+\varphi(a^*(x))-r^*(x), \, \  x \in (0,\hat b),
 \end{eqnarray}   
 and the SDE 
 $$ dJ_t^{Ag}=\left({{\lambda}} J_t^{Ag}-U(r^*(J_t^{Ag}))+h(a^*(J_t^{Ag}))-Z_t\frac{\varphi(a^*(J_t^{Ag}))}{\sigma}\right)dt+Z_tdW_t, \, \quad {  J_0^{Ag}\geq \underline{x}}$$
  admits a unique solution $\widehat{J_t^{Ag}}$.
   We define 
 \begin{equation}\label{tauverification}
\tau^* := ~\inf\{t\geq 0~:~ w(\widehat{J_t^{Ag}})\leq -{   U^{-1}}(\widehat{J_t^{Ag}})\}
\end{equation}
 and we assume that $( r^*(\widehat{J^{Ag}}) ,\tau^*, a^*(\widehat{J^{Ag}}))$ lies in  ${\cal{G}}$ and {${\esp^\pr[e^{  (\rho-2\lambda)\tau^*}\widehat{J_{\tau^*}^{Ag}}^{2}\un_{\{\tau^*<\infty\}}]}<\infty.$}\\
Then  we have
\begin{itemize}
\item [(a)] 
$w=v^{aux},$ and $\tau^*$ is an optimal stopping time for the problem (\ref{valuefunction}).\\
\item [(b)] The optimal rent is given by $r^*(x)=(U')^{-1}(-\frac{1}{w'(x)})\un_{w'(x)<0}$ {for all $x\in(0,\hat b)$}.
\end{itemize}
\end{itemize}
\end{theoreme}}
\noindent The proof is postponed in the Appendix.

\subsection{Comparison of the filtrations $\ff$ and $\ff^X$ at the optimum}
In this subsection, we study the inclusion properties of three filtrations: $\ff$  the filtration generated by {the} standard Brownian motion $W$ ({global} filtration), $\ff^X$  the filtration generated by the output  process $X$ (filtration of available information for the Principal) and $\ff^{J^{Ag}}$  the filtration generated by the process $J^{Ag}$. We have the natural inclusions: $\ff^{J^{Ag}}$ and $\ff^X$ are included in $\ff$. The following proposition shows that under some sufficient conditions, the  three filtrations coincide at the optimum, and in particular the process $J^{Ag}(x,R,\tau,A^*(Z))$ is $\ff^X$-adapted.
\begin{proposition}\label{filtracoinc}
We assume that the stochastic differential equation
 { \begin{equation}\label{SDES}
d{\bf J}_t^{Ag}={{\lambda}} {\bf J}_t^{Ag} dt-\Bigg(U(r^*({\bf J}^{Ag}_t))-h(a^*({\bf J}_t^{Ag}))\Bigg)dt
+\sigma\dfrac{h'}{\varphi'}(a^*({\bf J}_t^{Ag}))dW_t,~{\bf J}_0^{Ag}=x,
\end{equation}}
 admits a unique strong solution, where $ {\bf{J}}^{Ag}_t:=J_t^{Ag}(x,R,\tau,A^*(Z))$ for all $t\geq 0$, and that the function $x\longrightarrow a^*(x)$ is positive. Then, under Assumption \ref{hyp},
 the filtrations $\ff^X,$ $\ff^{{\bf J}^{Ag}}$ and $\ff$ coincide at the optimum.
\end{proposition}
\begin{proof}
The filtrations $\ff^X$ and $\ff^{{\bf J}^{Ag}}$ are included in $\ff$. As ${\bf J}^{Ag}$ is solution of the SDE(\ref{SDES}) then, $\ff^{{\bf J}^{Ag}}\subset\ff.$
As  $\sigma>0$, the function $a^*$ is positive and using Assumption \ref{hyp} $\sigma\frac{h'}{\varphi'}(a^*({\bf J}^{Ag}_t))$ is positive. By SDE (\ref{SDES}), we have
{ \begin{eqnarray*}
dW_t=\frac{1}{\sigma\frac{h'}{\varphi'}(a^*({\bf J}^{Ag}_t))}\Bigg[d{\bf J}_t^{Ag}-\Bigg(\lambda {\bf J}_t^{Ag}-U(r^*({\bf J}_t^{Ag}))
+h(a^*({\bf J}_t^{Ag}))-\sigma\frac{h'}{\varphi'}(a^*({\bf J}^{Ag}_t)))\Bigg)dt\Bigg].
\end{eqnarray*}}
Therefore we have $\ff\subset\ff^{{\bf J}^{Ag}}$ and the filtrations generated by $W$ and ${\bf J}^{Ag}$ coincide ($\ff\equiv\ff^{{\bf J}^{Ag}}$).
By definition of the { output process}, we have
  $$dX_t=\varphi(a^*({\bf J}_t^{Ag})dt+\sigma dW_t,~X_0~\mbox{is given.}$$
 Since $\sigma>0,$ we obtain
\begin{equation}\label{SDEW}
dW_t=\frac{1}{\sigma}[dX_t-\varphi(a^*({\bf J}_t^{Ag}))dt].
\end{equation} 
Furthermore, we have
\begin{eqnarray}\label{VC}\nonumber
d{\bf J}_t^{Ag}&=&\lambda {\bf J}_t^{Ag} dt-\Bigg(U(r^*({\bf J}_t^{Ag}))-h(a^*({\bf J}_t^{Ag}))+\sigma\frac{h'}{\varphi'}(a^*({\bf J}^{Ag}_t)))\Bigg)dt\\
&+&\frac{h'}{\varphi'}(a^*({\bf J}_t^{Ag}))\Bigg[-\varphi(a^*({\bf J}_t^{Ag}))dt
+dX_t\Bigg],\\
{\bf J}_0^{Ag}&=&x.\nonumber
\end{eqnarray}
Then the process $X$ appears as the unique source of noise driving (\ref{VC}). From (\ref{SDES}) and (\ref{SDEW}), the SDE (\ref{VC}) admits a unique solution, then ${\bf J}^{Ag}$ is $\ff^X$-adapted, hence $\ff^{{\bf J}^{Ag}}\subset \ff^X$.  
Therefore the three filtrations $\ff,~\ff^X$ and ${\bf J}^{Ag}$  coincide at the optimum.
\end{proof}\pf

\begin{proposition}
Under the assumptions of Proposition \ref{filtracoinc}, the value functions $v$ and $v^{aux}$ coincide.
\end{proposition}
\begin{proof}
We solved the auxiliary value function of the Principal defined by (\ref{valuefunction}) in the filtration $\ff$.  Since the filtration $\ff^X$ is included in $\ff$,  we have  $v\leq v^{aux}$.  Thanks to the Proposition \ref{filtracoinc}, we showed that the large filtration $\ff$ coincides with the filtration generated by the output process $\ff^X$ at the optimum. One conclude that $v = v^{aux}$ and  the initial problem of the Principal (\ref{valuefunctiono}) is solved.
\end{proof}\pf
\begin{remarque}
 Under the strong formulation approach, we  solved the original non-standard stochastic control problem under the assumptions that 
 the derivative of the marginal impact of the effort $\varphi$, the derivative of the cost of the effort $h$, and the diffusion term of the output process are positive. In this case, we obtain the same results as in the weak formulation approach.
 One advantage of the strong approach is to relax  the integrability conditions on the utility function, the marginal impact of the effort and the cost of the effort that are needed in the strong approach  due to the change of probability  from $\pr$ to $\pr^A.$ 
\end{remarque}

\section{Numerical study}  
For the numerical study, we choose the following functions for the first-best and second-best cases (see  Example \ref{exnum}):
\begin{itemize}
\item the impact of the effort on the output process : $\varphi(x)=3(1-\exp(-\alpha x)) ;\quad\alpha=0.1$
\item the cost of effort $h(x)=\exp(\beta x)-1; \quad \beta=0.1$
\item the Agent's utility  $U(x)=x^{\frac{1}{4}}$
\item The preference parameters for the Principal $\delta=0.08$ and  for the Agent $\lambda=0.2$. 
\end{itemize}
\subsection{First-best case}
In the first-best case,  the optimal rent is given by $R_t^*=(U')^{-1}(\frac{1}{\lambda_{Lag}}e^{(\lambda-\delta)t})$ and the optimal effort is given by  $A^*_t=(\frac{h'}{\varphi'})^{-1}(\frac{1}{\lambda_{Lag}}e^{(\lambda-\delta)t})\vee 0$, where $\lambda_{Lag}$ is solution of the equation
\begin{equation}\label{cas1}
\int_0^\infty e^{-{{{\lambda}} s}}
  \left(
  U\left((U')^{-1}(\frac{1}{\lambda_{Lag}}e^{(\lambda-\delta)s})\right)
  -h\left((\frac{h'}{\varphi'})^{-1}(\frac{1}{\lambda_{Lag}}e^{(\lambda-\delta)s} ) \vee 0\right)
  \right)ds=x.
  \end{equation}
To determine the domain $[0,x_{max}]$ on which  the Principal  proposes a contract to the Agent (that we call simply the continuation region by analogy with the second best-case),    we  compute first $\lambda_{Lag}(x)$ solution of (\ref{cas1}), then $H_s(x)=\varphi(A_s(x))-R_s(x),$ and   finally $\int_0^\infty e^{-\delta s}H_s(x)ds$. As $x \rightarrow H_s(x)$ is decreasing,  $x_{max}$ is the solution $\int_0^\infty e^{-\delta s}H_s(x)ds = 0$: if the Agent's reservation value is greater than $x_{max}$, then the Principal's value function would be negative if he proposes such a contract to the Agent.  We find that the continuation region is equal to {$[0, 5.45]$ (see Figure \ref{valuefirst2}).  \\
Figure \ref{LAGcas1} represents the variation of the Lagrange multiplier as a function of $x$. We recall that on  $[0,3\frac{\alpha}{\beta}]$, $\lambda_{Lag}$ is solution to 
$$\dfrac{c}{(pc\lambda_{Lag})^{\frac{p}{p-1}}(\lambda-\frac{p}{p-1}(\lambda-\delta))}+\frac{1}{\lambda}- \left(\frac{3\alpha}{\beta\lambda_{Lag}}\right)^{\frac{\beta}{\alpha+\beta}}\dfrac{\alpha+\beta}{\lambda\alpha+\beta\delta}= {x} $$
and on $[3\frac{\alpha}{\beta},\infty[, \lambda_{Lag}$ is solution to  
$$\dfrac{c}{(pc\lambda_{Lag})^{\frac{p}{p-1}}(\lambda-\frac{p}{p-1}(\lambda-\delta))}
+ \left( \frac{1}{\lambda}\left(\frac{3\alpha}{\beta \lambda_{Lag}}\right)^{\frac{\lambda}{\lambda-\delta}}
-\dfrac{\alpha+\beta}{\lambda\alpha+\beta\delta} 
\left(\frac{3\alpha}{\beta\lambda_{Lag}}   \right)^{\frac{\lambda}{\lambda-\delta}  } \right)  = x.$$
In this numerical study $\lambda_{Lag}=3\frac{\alpha}{\beta}=3$   corresponds to $x=1.64$.

\begin{figure}[H]
\begin{center}
\centering \includegraphics[width=10cm,height=8cm]{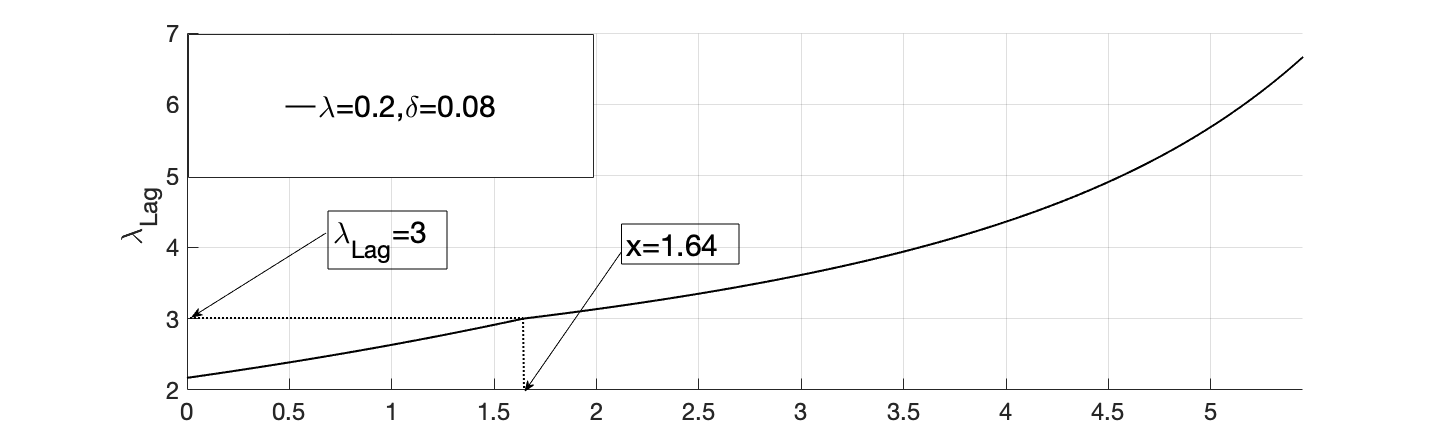}
\caption{The variation of the Lagrange multiplier as a function of $x.$ }\label{LAGcas1}
\end{center}
\end{figure}

\noindent Figure \ref{valuefirst2} represents the value function for the Principal. The discontinuity in the slope of the curve at $x=1.64$ corresponds to the $x$-value  for which $\lambda_{Lag} =  \frac{3\alpha}{\beta}=3.$
\begin{figure}[H]
\begin{center}
\centering \includegraphics[width=12cm,height=8cm]{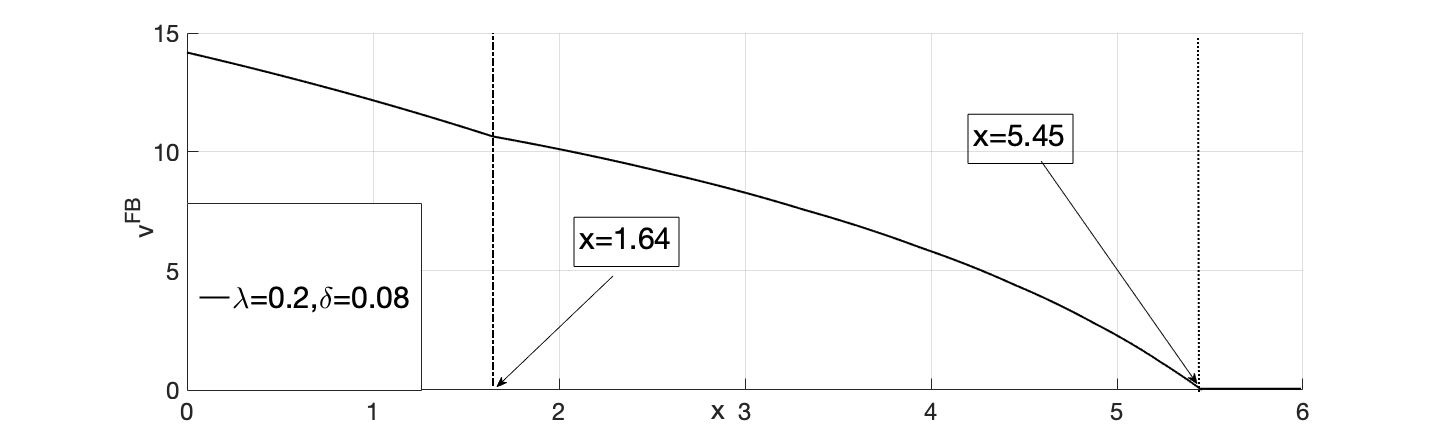}
\caption{ Value Function (first-best) in $[0,6]$. }\label{valuefirst2}
\end{center}
\end{figure}

\newpage
\noindent Figures \ref{renteffortfirstbis1} and \ref{renteffortfirstbis2} represent the optimal rent paid by the Principal, as a decreasing function of the effort provided by the Agent, for two different values of $t=0$ and $t=25$. \\
\begin{minipage}{0.4\textwidth}
    \begin{figure}[H]
        \includegraphics[width=8cm,height=7cm]{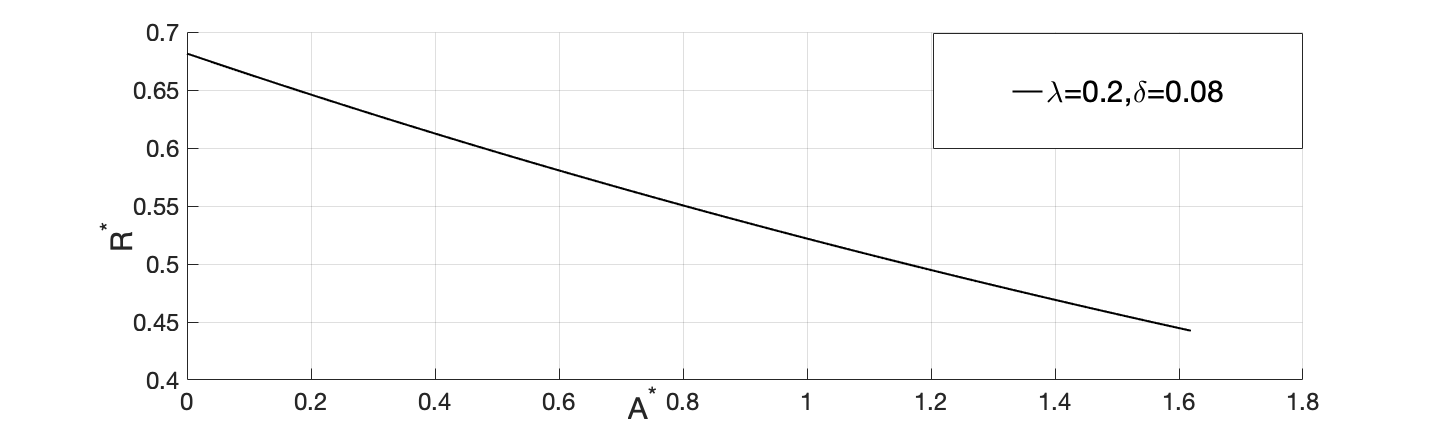}
      \caption{ Optimal rent  as a function of the effort  for $t=0$ (first-best case)} \label{renteffortfirstbis1}
    \end{figure}  
\end{minipage}
\hspace{2ex}
\begin{minipage}{0.4\textwidth}
    \begin{figure}[H]
        \includegraphics[width=8cm,height=7cm]{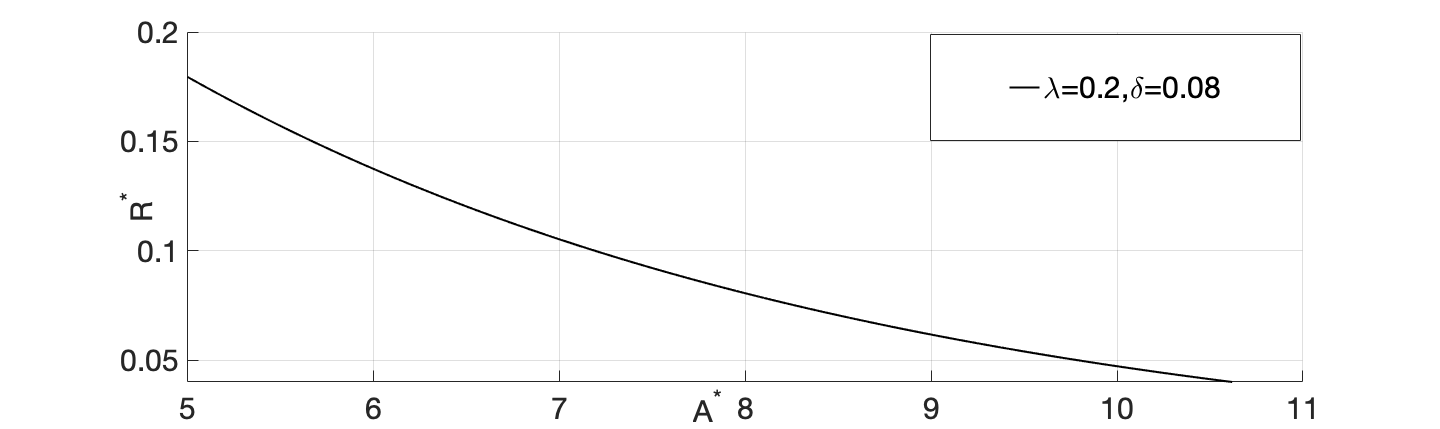}
      \caption{ Optimal rent  as a function of the effort  for $t=15$ (first-best case)}  \label{renteffortfirstbis2}
    \end{figure}  
\end{minipage}
\\

\noindent Figure \ref{rentcas1} (resp. Figure \ref{effortcas1}) represents the optimal rent (resp. optimal effort) as a function of $x$ an $t$.  The optimal rent is decreasing in $t$ (since the Agent is more impatient than the Principal) and increasing in $x$. Besides, the optimal effort  is increasing in $t$ and decreasing in $x$. 

\begin{minipage}{0.4\textwidth}
    \begin{figure}[H]
        \includegraphics[width=8cm,height=8cm]{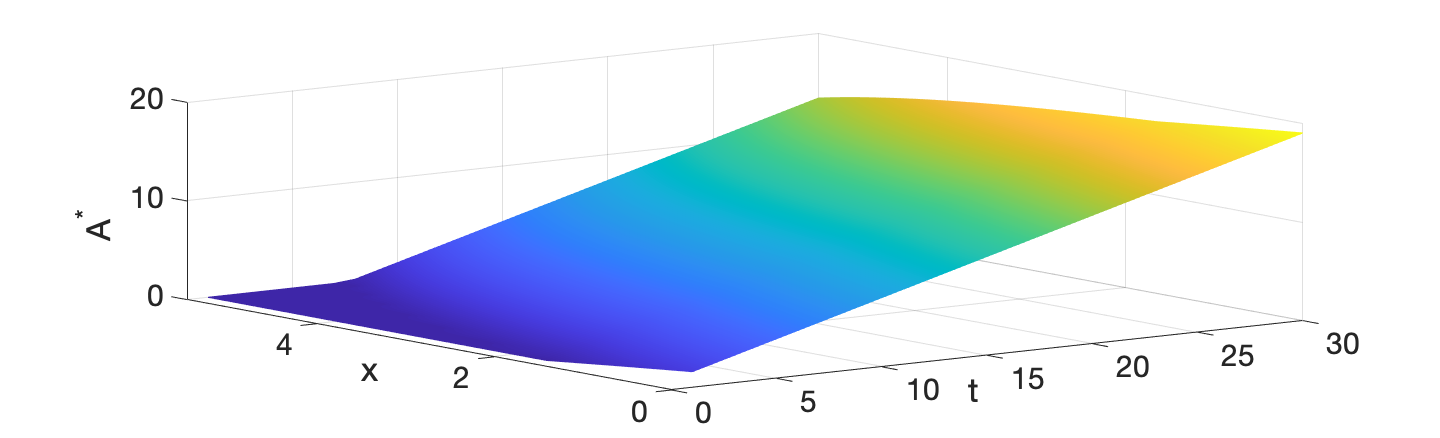}
      \caption{The optimal effort.}
    \end{figure}
\end{minipage}
\hspace{2ex}
\begin{minipage}{0.4\textwidth}
    \begin{figure}[H]
        \includegraphics[width=8cm,height=8cm]{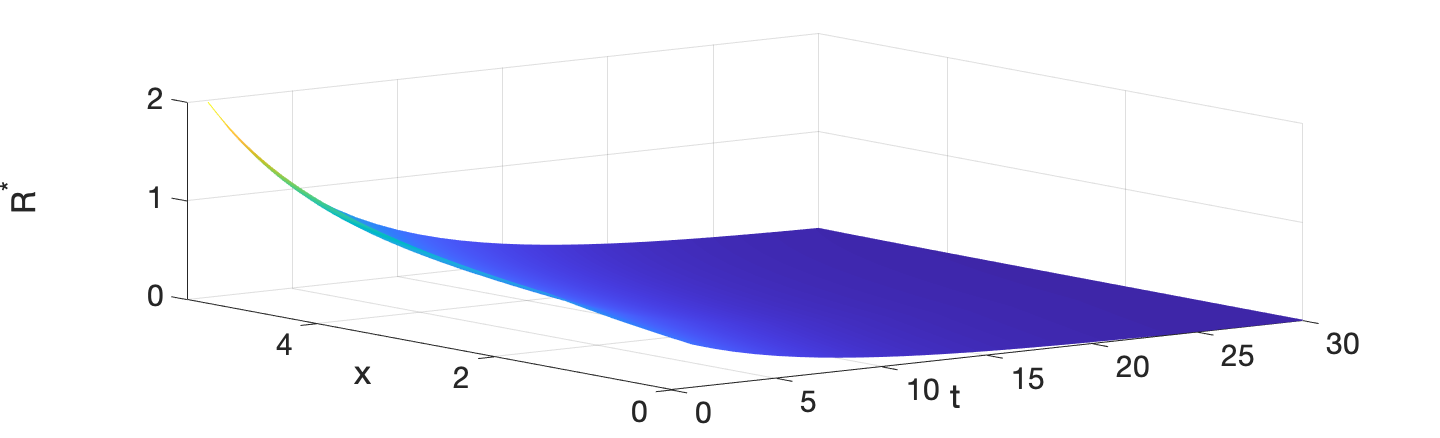}
      \caption{The optimal rent. }
    \end{figure}
\end{minipage}

\newpage
\noindent Figures \ref{valuefirst}, \ref{rentcas1}  and \ref{effortcas1} below provide a focus  of the value function and optimal effort/rent on  a smaller interval $x \in [0,0.95]$ to stay on a similar interval than in the second-best case (see the continuation region in the second-best case in Section \ref{sec:numerics2}).

\begin{figure}[H]
\begin{center}
\centering \includegraphics[width=12cm,height=8cm]{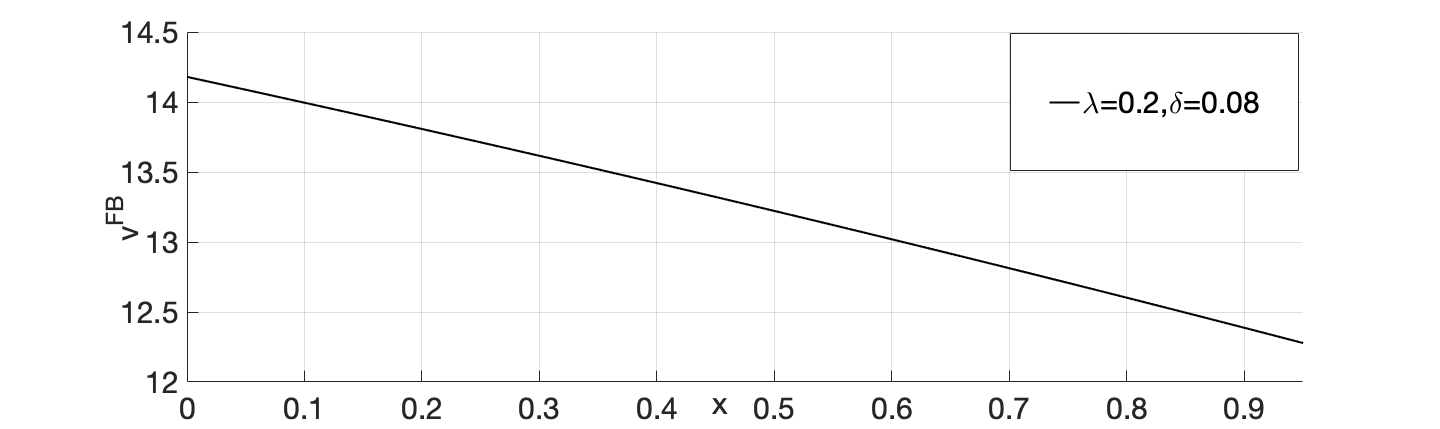}
\caption{ Value Function (first-best): zoom on $[0,0.95]$. }\label{valuefirst}
\end{center}
\end{figure}

\begin{minipage}{0.4\textwidth} 
    \begin{figure}[H]
        \includegraphics[width=8.5cm,height=8cm]{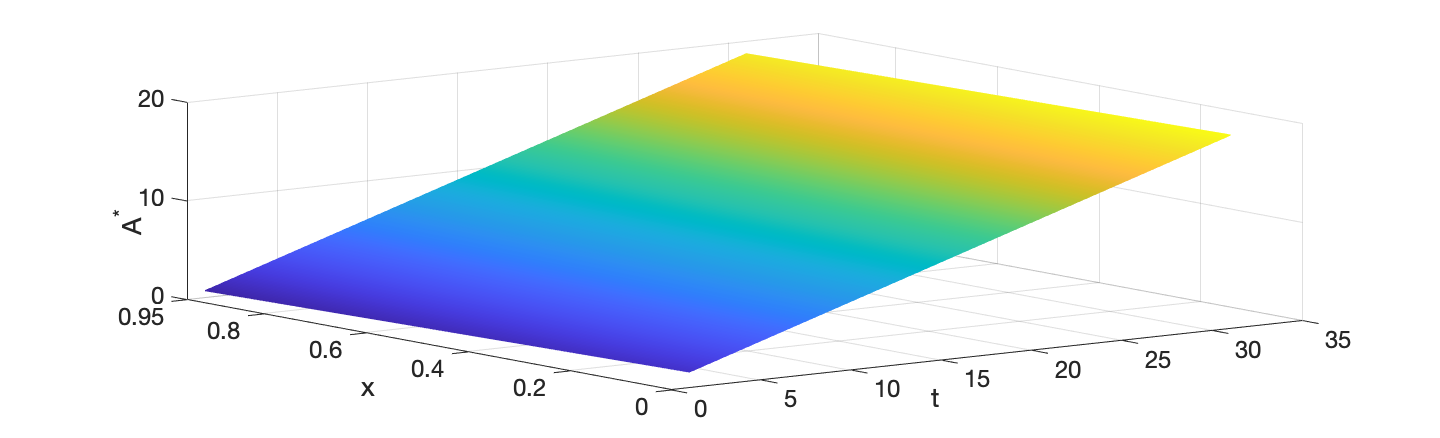}
      \caption{The optimal effort (first-best): zoom on $[0,0.95]$. }\label{effortcas1}
    \end{figure}
    \end{minipage}
    \hspace{2ex}
    \begin{minipage}{0.4\textwidth}
     \begin{figure}[H]
        \includegraphics[width=8.5cm,height=8cm]{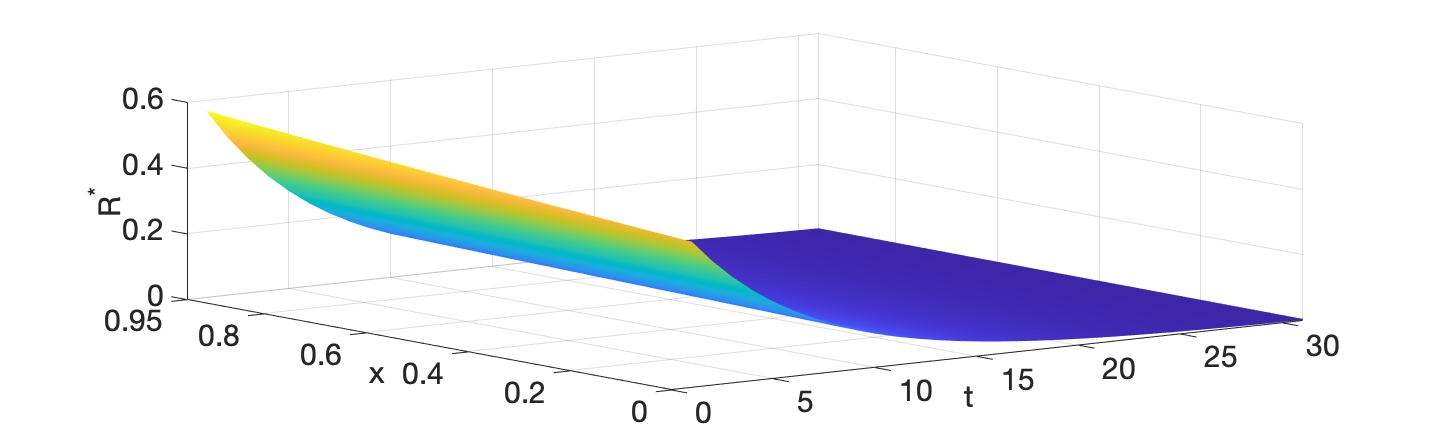}
      \caption{The optimal rent (first-best): zoom on $[0,0.95]$. }\label{rentcas1}
    \end{figure}
\end{minipage}

\hspace{2ex}

\subsection{Second-best case}\label{sec:numerics2}
We approximate numerically the solution of the HJBVI (\ref{IVHJB-3}) by using a policy iteration algorithm named Howard algorithm.
 The numerical approximation of the  solution of (\ref{IVHJB-3}) consists in three steps (for more details see  Hajjej et al.  \cite{hajjej2019optimal}):
\begin{enumerate}
\item Reduction to a bounded domain.  We have to replace    $[0,\infty)$  by a bounded domain $[0,\overline{x}]$.   Since the behavior of the HJB solution at $\infty$  is known, $v(x)=-U^{-1}(x)$ for $x$ large enough, and we take this   boundary condition. The  choice  of the boundary $\overline{x}$ is empirical 	and the robustness is studied by varying $\overline{x}$.
 \item We use  finite difference approximations to discretize the variational inequality (\ref{IVHJB-3}).
 \item  We use Howard algorithm (see Howard \cite{howard1960dynamic}) to solve the discrete equation.
 \end{enumerate}
  In this section we use the same functions $\varphi, h$ and $ U$ as in the first-best case. We first take the following value for the  volatility  $\sigma=1.85$. We observe in Figures \ref{LAG1} and \ref{LAG2} that the value function is concave,  in accordance with Sannikov \cite{sannikov2008continuous}. The continuation region is (0, 0.32) on which the value function is strictly concave, then it is equal to $-U^{-1}(x)$.  \\
 
\begin{minipage}{0.4\textwidth}
    \begin{figure}[H]
        \includegraphics[width=8cm,height=8cm]{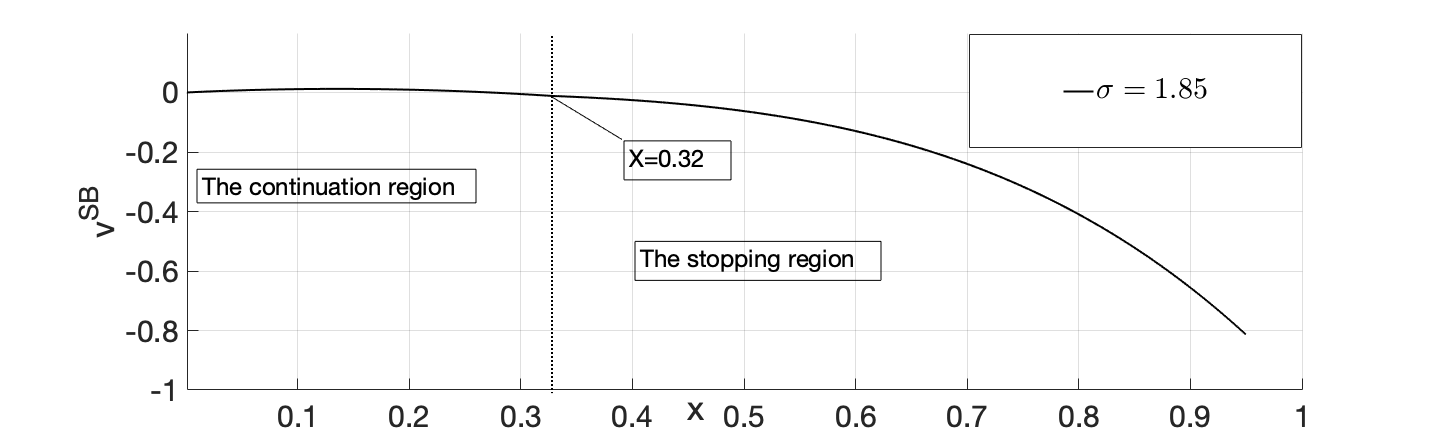}
\caption{Value function (second-best) in $[0,\Bar{x}]$ }\label{LAG1}
    \end{figure}
\end{minipage}
\hspace{2ex} 
\begin{minipage}{0.4\textwidth}
    \begin{figure}[H]
        \includegraphics[width=8cm,height=8cm]{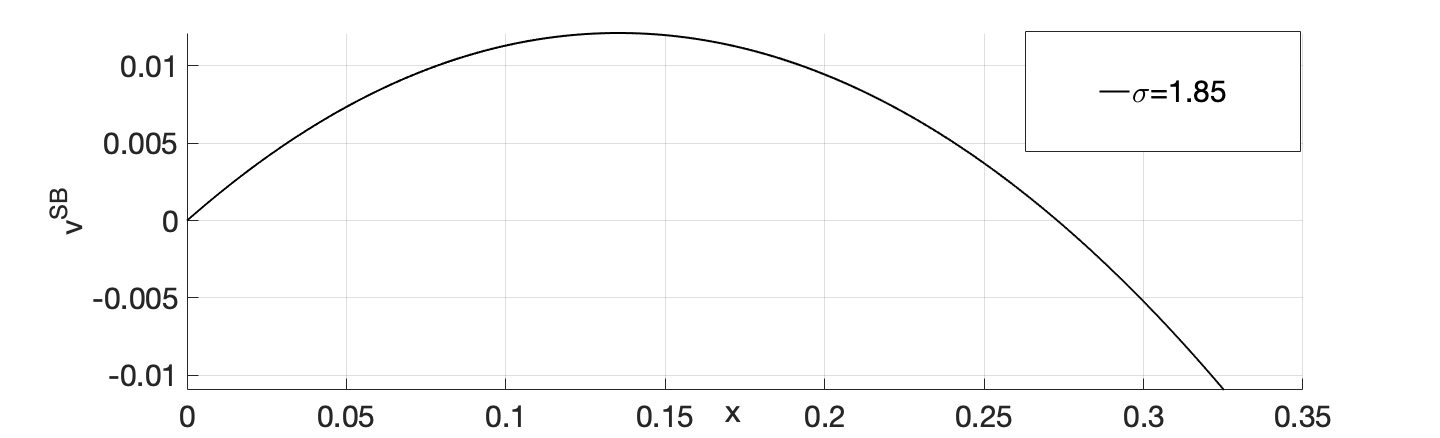}
\caption{Value function (second-best) on the continuation region}\label{LAG2}
    \end{figure}
\end{minipage}

\newpage
\noindent Figure \ref{sigma} computes the value function for different value of $\sigma$. The higher $\sigma$, the smaller the value function for the Principal. 
These numerical results are in accordance with the ones obtained in \cite{hajjej2019optimal} established in a weak approach. 

    \begin{figure}[H]
    \begin{center}
\centering 
         \includegraphics[width=9cm,height=8cm]{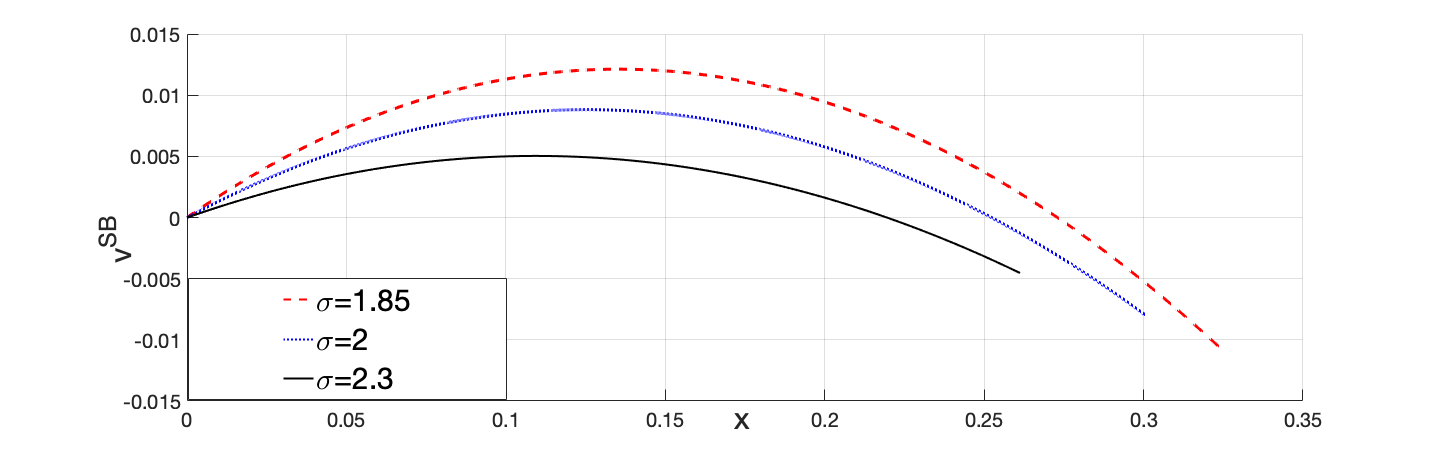}
\caption{Value function for different $\sigma$ in the second-best case}\label{sigma}
       \end{center}
    \end{figure}

\subsection{The value of the information}
Let us first compare the optimal effort (Figures \ref{effortfirst}-\ref{effortfirst15}-\ref{effortsecond}) and the optimal rent  (Figures \ref{rentfirst}-\ref{rentfirst15}-\ref{rentsecond}) in the first-best  and second-best cases.  For both cases, the higher is $x$ the reservation value for the Agent, the fewer effort the Agent will provide, and the higher the rent he will receive. Roughly speaking, when the Agent is richer, he is less motivated to provide  effort, and the Principal should pay him more to prompt him to give more effort. If we focus on the first-best case, 
Figures \ref{effortfirst} and \ref{rentfirst} correspond to the optimal effort and optimal rent  at time $t=0$: since the Agent is more impatient than the Public, the Principal agrees  to pay  a higher rent and to receive a lower effort from the Agent at the beginning of the contract, compared to second-best case. As shown in the previous Figures  \ref{effortcas1} and \ref{rentcas1}, as well as in Figures \ref{effortfirst15} and \ref{rentfirst15},   the rent will then rapidly decrease in time,   and the effort will increase. In addition the contract in the first-best case is perpetual, while it may stop at a stopping time in the second best case, such that the first-best case is ultimately  much more favorable for the Principal than the second-best case, as shown in Figure \ref{valuefinformation1}.  

\begin{minipage}{0.4\textwidth}
    \begin{figure}[H]
        \includegraphics[width=8cm,height=8cm]{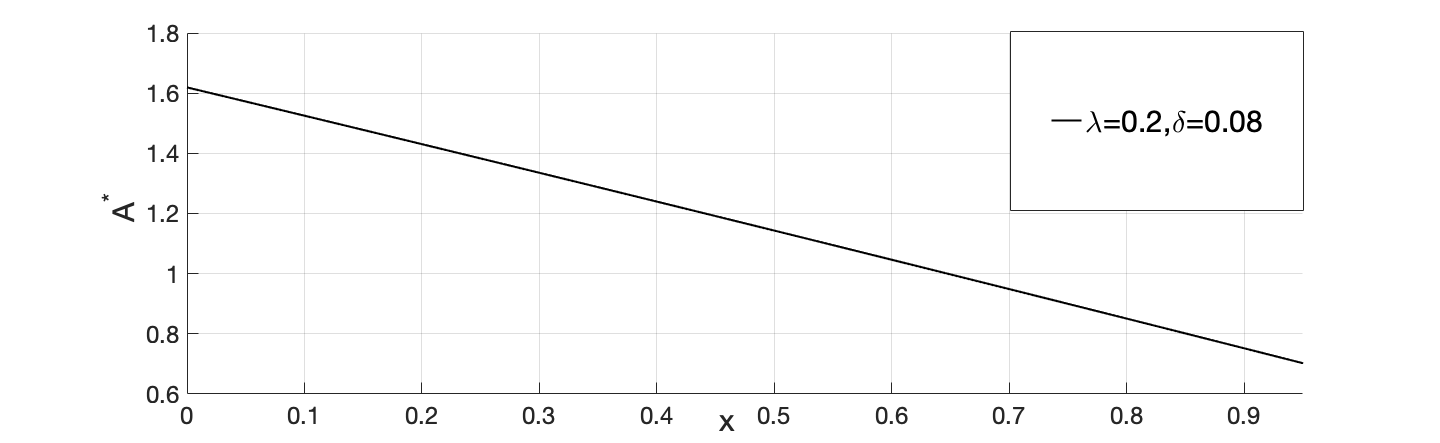}
       \caption{  Optimal effort (first-best) for $t=0$.}\label{effortfirst}
    \end{figure}
\end{minipage}
\hspace{2ex} 
\begin{minipage}{0.4\textwidth}
     \begin{figure}[H]
        \includegraphics[width=8cm,height=8cm]{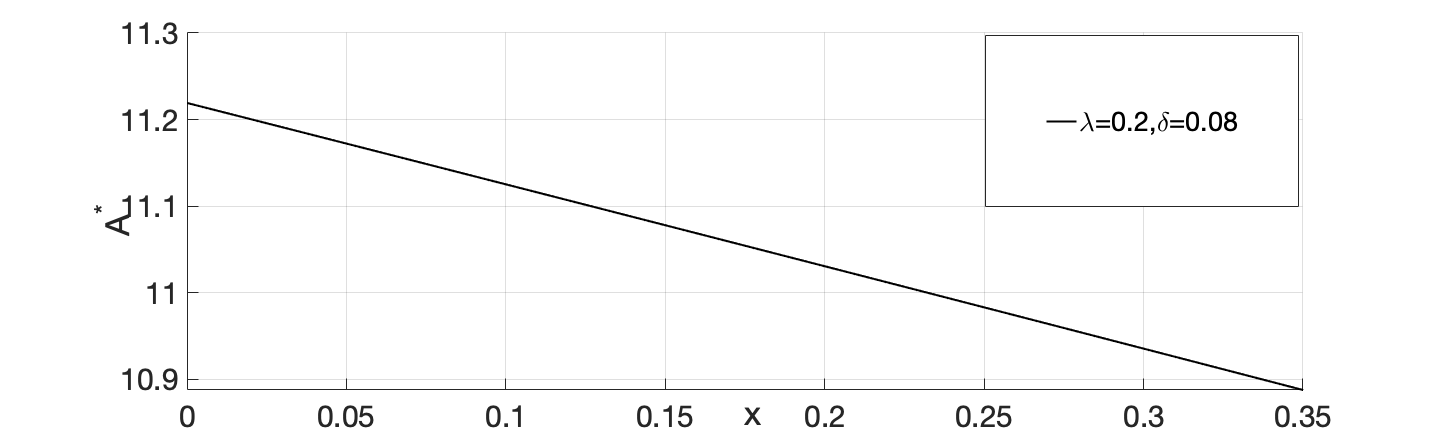}
       \caption{ Optimal effort (first-best) for $t=15$.}\label{effortfirst15}
    \end{figure}
\end{minipage}

    \begin{figure}[H]
    \begin{center}
        \includegraphics[width=8cm,height=8cm]{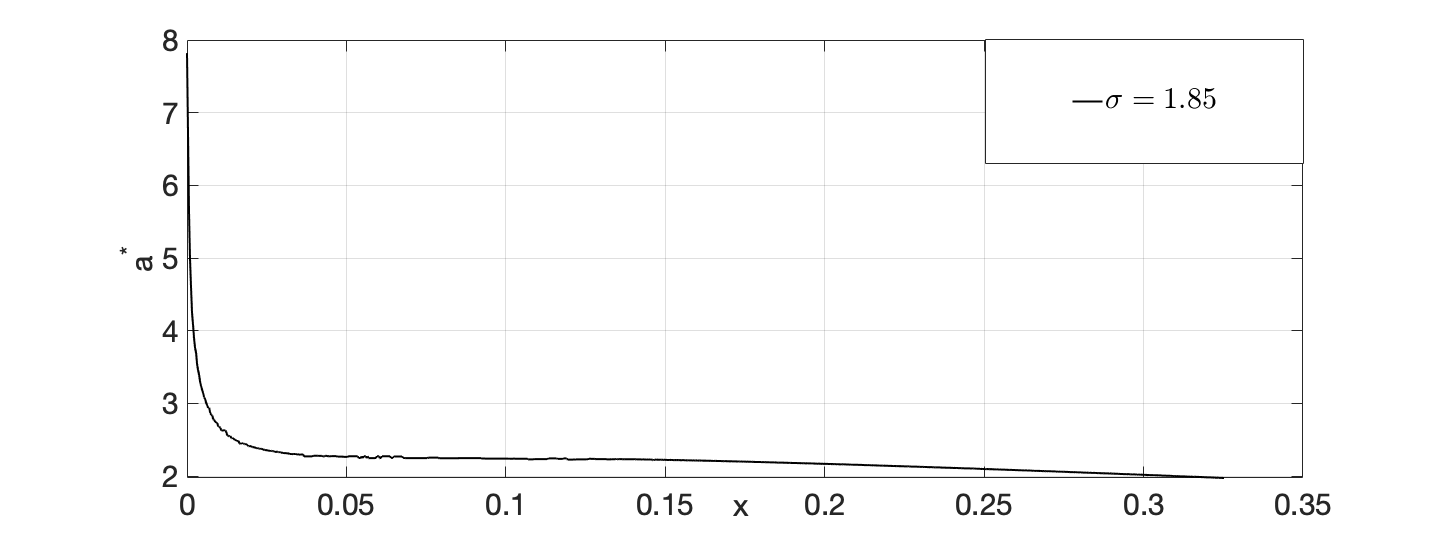}
       \caption{Optimal effort (second-best) }\label{effortsecond}
        \end{center}
    \end{figure}

\begin{minipage}{0.4\textwidth}
    \begin{figure}[H]
        \includegraphics[width=8cm,height=8cm]{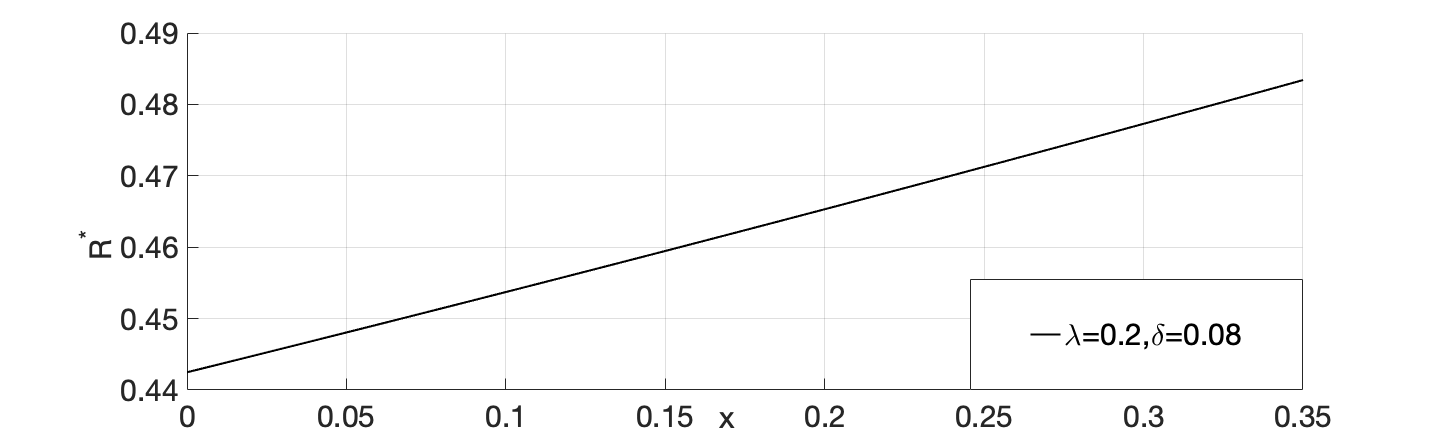}
       \caption{ Optimal rent(first-best) for $t=0$. }\label{rentfirst}
    \end{figure}
\end{minipage}
\hspace{2ex} 
\begin{minipage}{0.4\textwidth}
    \begin{figure}[H]
        \includegraphics[width=8cm,height=8cm]{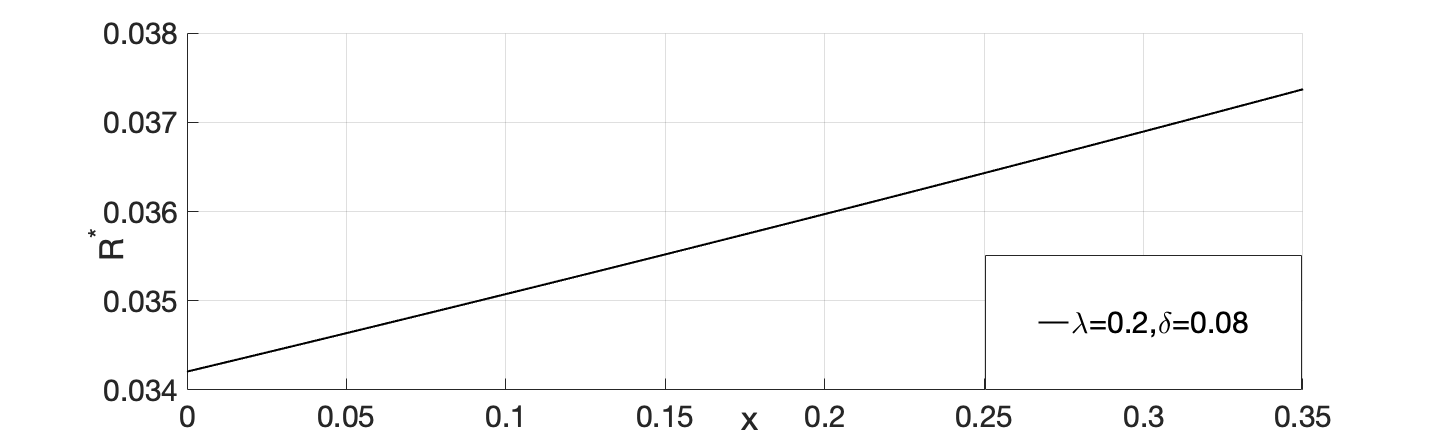}
       \caption{ Optimal rent(first-best) for $t=15$. }\label{rentfirst15}
    \end{figure}
\end{minipage}

    \begin{figure}[H]
    \begin{center}
        \includegraphics[width=9cm,height=8cm]{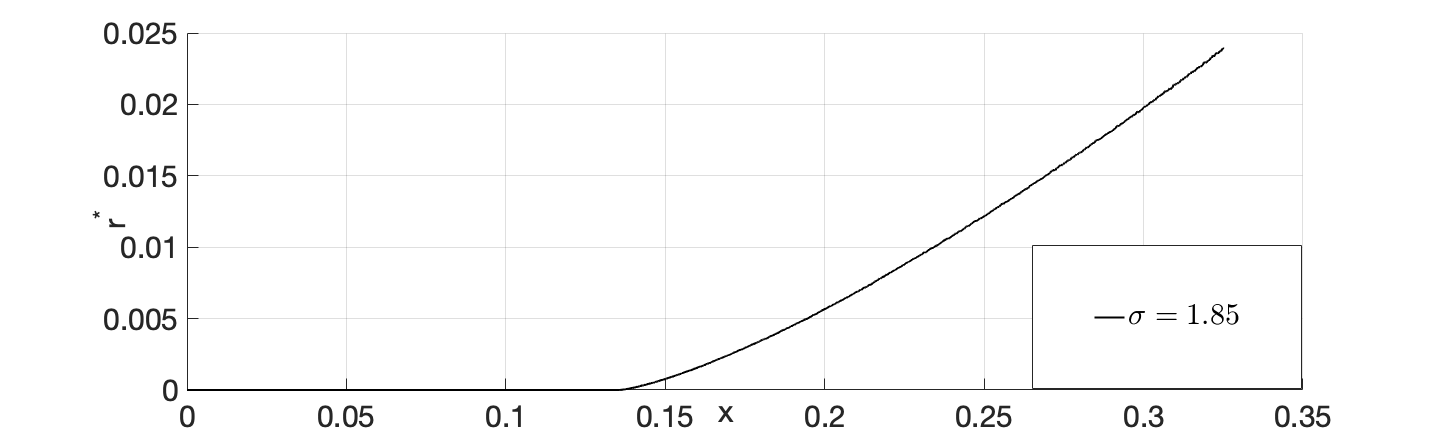}
       \caption{Optimal rent (second-best) }\label{rentsecond}
       \end{center}
    \end{figure}

\newpage

\noindent Let us now compare in Figure \ref{valuefinformation1}   the Principal  value function in the first-best and in the second-best cases. As the Principal is risk-neutral, 
the difference represents the "value of information", computed in  Figure \ref{valueofinformation}. 
The value function for the Principal in the first-best  case is greater than the value function for the Principal in the second-best case.  The 
 continuation region  in  the first-best case ($[0,5.45]$) is much larger than in the second-best ($[0,0.32]$), thanks to the risk sharing between the  Principal and Agent.   The value function in the first-best {case} is equal to zero in the stopping region ({$\tau=0$ is optimal}), but in second-best {case} the value function in the stopping region is negative equal to $-U^{-1}(x)$.  We focus below on the region $x\in [0,0.95]$.  The value of the information is a convex function of $x$. This, together with Figure \ref{sigma}, shows  the higher the risk, the more  the value of the  information.

\begin{minipage}{0.4\textwidth}
    \begin{figure}[H]
        \includegraphics[width=8cm,height=8cm]{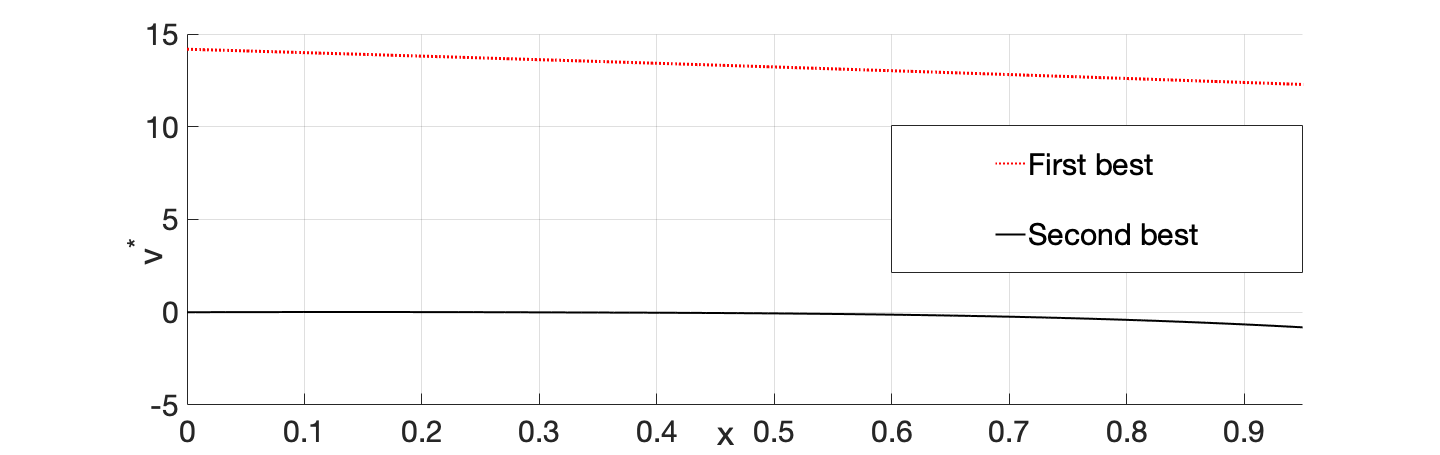}
      \caption{Value Function in the first-best and in the second-best. }\label{valuefinformation1}
    \end{figure}
\end{minipage}
\hspace{2ex} 
\begin{minipage}{0.4\textwidth}
    \begin{figure}[H]
        \includegraphics[width=9cm,height=8cm]{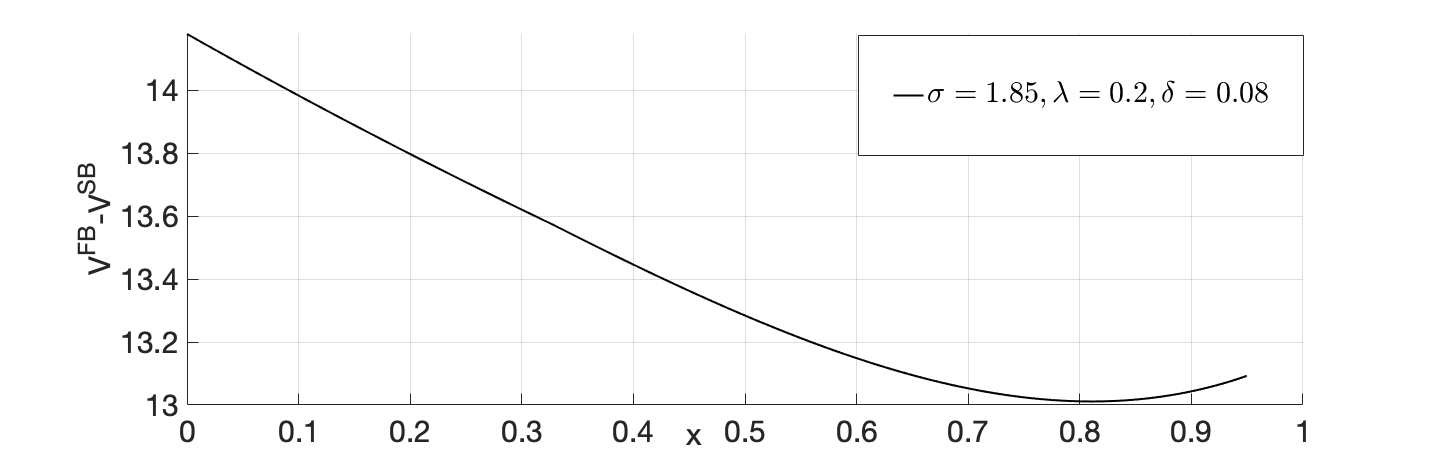}
      \caption{The value of information for $\sigma=1.85$. }\label{valueofinformation}
    \end{figure}
\end{minipage}

\section{Appendix}
{\bf Proof of Theorem   \ref {verification}}

  $[(a)]$
  {\underline {First step: We prove
      $v^{aux}\geq w$ on $\R^+$}}. Let $x\geq0$,  $n\in\N,$  and  an admissible contract $(R,\tau,A^*(Z))\in\mathcal G$. If $x=0$, then from assumption(i), we have $v(0)=w(0)=0$.
  We assume that $0<x$.
  We introduce the following  $\ff$-stopping time:
  $$\tau_n~:=~ \tau\wedge \inf\{t: |w'(J_t^{Ag}(x,R,\tau,A^*(Z))){ \sigma\frac{h'(A^*(Z_t))}{\varphi'(A^*(Z_t))}}	| > n\}.$$
{{From (i)-(ii),
 $w$ is continuous on $\R_+$, $w\in C^2([0,\hat b))$ and  $w=-U^{-1}\in C^2([\hat b,\infty))$, then $w$ is continuous and piecewise $C^2$ on $\R_+$.
Applying the generalized It\^o's formula (see Krylov \cite{krylov2008controlled}, Theorem 2, p. 124) between time $0$ en $\tau_n$ to the process 
$\left(e^{-\delta t}w(J_t^{Ag}(x,R,\tau,A^*(Z)))\right)_{t\geq 0}$ {\footnotesize 
\begin{eqnarray*}
w(x)&=&e^{-\delta\tau_n}w(J_{\tau_n}^{Ag}(x,R,\tau,A^*(Z)))-\int_0^{\tau_n}e^{-\delta s}[-\delta w(J_s^{Ag}(x,R,\tau,A^*(Z)))\\
  &+&{\cal{L}}^{A^*(Z_s),R_s}w(J_{s}^{Ag}(x,R,\tau,A^*(Z)))]ds+\int_0^{\tau_n}e^{-\delta s}w'(J_s^{Ag}(x,R,\tau,A^*(Z)))
\sigma\frac{h'(A^*(Z_s))}{\varphi'(A^*(Z_s))}{\un_{\{A^*(Z_s)>0\}}}dW_s.
\end{eqnarray*}}
Taking the expectation, we obtain:
\begin{eqnarray}\label{Ito2-4}
w(x)&=&\esp\Bigg[e^{-\delta \tau_n}w(J_{\tau_n}^{Ag}(x,R,\tau,A^*(Z)))-\int_0^{\tau_n}e^{-\delta s}[-\delta w(J_{s}^{Ag}(x,R,\tau,A^*(Z)))\nonumber \\
&+& {\cal{L}}^{A^*(Z_s),R_s}w(J_{s}^{Ag}(x,R,\tau,A^*(Z)))]ds\Bigg]\nonumber\\ 
&\geq& \esp\left[e^{-\delta\tau_n}w(J_{\tau_n}^{Ag}(x,R,\tau,A^*(Z)))+ \int_0^{\tau_n}e^{-\delta s}(\varphi(A^*(Z_s))-R_s)ds\right],
\end{eqnarray}
where the  inequality is obtained by using  assumptions(iii)-(iv).\\ 
By using the definition of {the} set $\mathcal{D}^{Ag}_\rho$, and Cauchy Schwarz inequality,  we have
\begin{eqnarray*}
  \sup_{n\in\N}\esp\big[\int_0^{\tau_n}e^{-\delta s}(\varphi(A^*(Z_s))-R_s)ds\big]
  \leq \frac{1}{\delta}\esp\big[\int_0^{\infty}e^{-\delta s}|\varphi(A^*(Z_s))|^2+|R_s|^2ds\big]<\infty.
\end{eqnarray*}
Therefore, $\Sup_{n\in\N}\esp\big[\int_0^{\tau_n}e^{-\delta s}(\varphi(A^*(Z_s))-R_s)ds\big]<\infty$, so we have $(\int_0^{\tau_n}e^{-\delta s}\varphi(A^*(Z_s)-R_s)ds)_{n\in\N}$ is uniformly integrable. Thus, we have the convergence in $\mathbb{L}^1$ and we may pass to the limit as  $n\rightarrow\infty$, and we get
\begin{equation}\label{inq1-3}\Lim_{n\rightarrow\infty}\esp[\int_0^{\tau_n}e^{-\delta
 s}(\varphi(A^*(Z_s))-R_s)ds]=\esp[\int_0^{\tau}e^{-\delta
 s}(\varphi(A^*(Z_s))-R_s)ds].\end{equation}
As $w$ satisfies the growth condition (\ref{eqcroissance}), we have
\begin{eqnarray*}
\esp\left[|e^{-\delta\tau_n}w(J_{\tau_n}^{Ag}(x,R,\tau,A^*(Z)))|\right]&\leq&\esp\left[|e^{-\delta\tau_n}K(1+U^{-1}(J^{Ag}_{\tau_n}(x,R,\tau,A^*(Z))))|\right]\\
\end{eqnarray*}
Using (\ref{hyp3}), \quad
$\sup_{n\in\N}\esp\left[|e^{-\delta\tau_n}w (J_{\tau_n}^{Ag}(x,R,\tau,A^*(Z)))|\right]
<\infty$\quad
and we may pass to the limit as $n\rightarrow\infty$, and we get
\begin{equation}\label{inq2-3}
\Lim_{n\rightarrow\infty}\esp[e^{-\delta\tau_n}w(J_{\tau_n}^{Ag}(x,R,\tau,A^*(Z)))]=E[e^{-\delta\tau}w(J_{\tau}^{Ag}(x,R,\tau,A^*(Z)))].\end{equation}
By (\ref{Ito2-4}), (\ref{inq1-3}) and (\ref{inq2-3}), we have
$$w(x)\geq \esp\left[ \int_0^{\tau}e^{-\delta s}(\varphi(A^*(Z_s))-R_s)ds+e^{-\delta\tau}w(J_{\tau}^{Ag}(x,R,\tau,A^*(Z)))\right].$$
From assumption (ii),  $w(J_{\tau}^{Ag}(x,R,\tau,A^*(Z)))\geq -U^{-1}(J_\tau^{Ag}(x,R,\tau,A^*(Z)))$, and  we deduce 
$$w(x)\geq \esp\left[ \int_0^{\tau}e^{-\delta s}(\varphi(A^*(Z_s))-R_s)ds-e^{-\delta\tau}J_{\tau}^{Ag}(x,R,\tau,A^*(Z))\right].$$
By taking the supremum, we obtain
$$w(x)\geq \Sup_{(R,\tau,A^*(Z))\in{\cal{G}}}\esp\left[ \int_0^{\tau}e^{-\delta s}(\varphi(A^*(Z_s))-R_s)ds-e^{-\delta\tau}J_{\tau}^{Ag}(x,R,\tau,A^*(Z))\right]=v^{aux}(x).$$

\noindent  {\underline {Second step: We prove $v^{aux}\leq w$ on $\R^+$}} If $x\geq \hat b$, then from assumption (ii) and the definition of the value function (\ref{valuefunction}),
we have $v^{aux}(x)\geq w(x)$. \\From now, we assume that $0< x < \hat b$. We now consider the feedback control $( r^*(\widehat{ J^{Ag}}) ,\tau^*, a^*(\widehat{ J^{Ag}}) )$ which is assumed to be in ${\cal{G}}$. Let $\tau^*$ be the stopping time introduced in (\ref{tauverification}). Then $\tau^*\in {\cal{T}}.$ We introduce the following $\ff$-stopping time:
$$\tau_n~:=~\tau^*\wedge \inf\{t:{ |w'(\widehat{ J_t^{Ag})}}{ \sigma\frac{h'(a^*(\widehat{ J_t^{Ag}})}{\varphi'(a^*(\widehat{ J_t^{Ag}})}})|> n\}.$$
Observe that  $w(\widehat{ J_t^{Ag}})>-U^{-1}(\widehat{ J_t^{Ag}})$ on $\lbr0,\tau_n{\lbr}\subset\lbr0,\tau^*\lbr$. 
Then on $\lbr0,\tau_n\lbr$, by (\ref{IVHJB-3})
$$\delta w({\widehat J^{Ag}}_t)- [{\cal L}^{a^*({\widehat J^{Ag}}_t),r^*({\widehat J^{Ag}}_t)} w({\widehat J^{Ag}}_t)+\varphi(a^*({\widehat J^{Ag}}_t))-r^*({\widehat J^{Ag}}_t)]=0.$$
Therefore
\begin{eqnarray*}
w(x)&=&\esp \left[e^{-\delta \tau_n}w({\widehat J}_{\tau_n}^{Ag})-\int_0^{\tau_n}e^{-\delta s}\big(-\delta w({{\widehat J^{Ag}}}_s)+{\cal{L}}^{a^*({ {\widehat J^{Ag}}}_s),r^*({\widehat J^{Ag}}_s)}w({{\widehat J^{Ag}}}_s)\big)ds\right]\\
&=&\esp\left[e^{-\delta \tau_n}w({\widehat J^{Ag}}_{\tau_n})+\int_0^{\tau_n}e^{-\delta s}(\varphi(a^*({\widehat J^{Ag}}_s))-r^*({\widehat J^{Ag}}_s))ds\right].\\
\end{eqnarray*}
Similarly to the first step, we show that {$(\int_0^{\tau_n} e^{-\delta s}(\varphi(a^*({\widehat J^{Ag}}_s))-r^*({\widehat J^{Ag}}_s))ds)_{n}$}  and {$\left(w(\widehat{ J^{Ag}_{\tau_n}})\right)_{n}$} are uniformly integrable. We may pass to the limit
$n\rightarrow\infty$,  $\tau_n \rightarrow \tau^*$ a.s,  and since $w({\widehat J^{Ag}}_{\tau^*})=-{\widehat J^{Ag}}_{\tau^*}$, we obtain 
\begin{eqnarray*}
w(x)&=&\esp\left[\int_0^{\tau^*} e^{-\delta s}(\varphi(a^*({\widehat J^{Ag}}_s))-r^*({\widehat J^{Ag}}_s))ds-e^{-\delta \tau^*}{\widehat J^{Ag}}_{\tau^*}\right]\\
&=&J^P_0( r^*({\widehat J^{Ag}}),\tau^*, a^*({\widehat J^{Ag}}             ) )\\
&{\leq}& v^{aux}(x).
\end{eqnarray*}
We conclude that $w =  v^{aux}$ on $\R^+$ and $(r^*({\widehat J^{Ag}}),\tau^*, a^*({\widehat J^{Ag}}  ))$ is an optimal feedback control.\\
$[(b)]$ For $x\in (0,\hat b)$, we maximize the function 
$$ f(x,.) :r\mapsto -w'(x)U(r)-r.$$
When $w'(x)\geq 0$,  the function $f(x,.)$ is non-increasing and the optimum is achieved for  $r=0.$\\
Otherwise, the function $f(x,.)$ is concave and the optimal rent is given by
 $r^*(x)=\argmax_{r} (f(x,r)).$ Therefore
$r^*(x)=(U')^{-1}(\frac{-1}{w'(x)})\un_{w'(x)<0}.$

\newcommand{\etalchar}[1]{$^{#1}$}

\end{document}